\documentclass[a4paper,10pt]{amsart}
\usepackage{amssymb,amsmath,amsfonts,amsthm}
\usepackage[dvips]{graphicx}
\usepackage{psfrag}

\theoremstyle{plain}
\newtheorem{main}{Theorem}

\newtheorem{theorem}{Theorem}[section]
\newtheorem{lemma}[theorem]{Lemma}
\newtheorem{proposition}[theorem]{Proposition}
\newtheorem{corollary}[theorem]{Corollary}
\theoremstyle{remark}
\newtheorem{remark}[theorem]{Remark}

\newtheorem{example}[theorem]{Example}

\newtheorem{conjecture}[theorem]{Conjecture}

\newcommand{\quand}{\quad\text{and}\quad}
\newcommand{\Leb}{\operatorname{vol}}

               \def\cal{\mathcal}
           \def\ea{\end{array}}
          \def\ec{\end{center}}
     \def\ed{\end{description}}
        \def\ee{\end{equation}}
       \def\eea{\end{eqnarray}}
     \def\eeaa{\end{eqnarray*}}
 \def\et{\end{thebibliography}}

\def\Diff{{\rm Diff}}

\def\supp{\operatorname{supp}}

\def\cPO{{\mathcal{V}}}

\def\cU{{\mathcal U}}
\def\cV{{\mathcal V}}

\def\cB{{\mathcal B}}
\def\cC{{\mathcal C}}

\def\cF{{\mathcal F}}

\def\cP{{\mathcal P}}

\def\cS{{\mathcal S}}
\def\cW{{\mathcal W}}

\def\cBk{{\mathcal B}^k}
\def\id{\operatorname{id}}

\def\Nc{N/\cW^c}
\def\dc{d_c}
\def\fc{f_c}

\def\vep{\varepsilon}

\def\TT{{\mathbb T}}
\def\RR{{\mathbb R}}

\def\ZZ{{\mathbb Z}}

\def\hW{\widehat{W}}

\title[Physical measures and absolute continuity]
{Physical measures and absolute continuity\\ for one-dimensional center direction}
\author{Marcelo Viana$^1$ and Jiagang Yang$^{2,1}$}
\date{\today}

\thanks{Partially supported by CNPq and PRONEX-Dynamical Systems.
        J.Y. was supported by scholarships from TWAS-CNPq and FAPERJ.}
\thanks{$^1$IMPA, Est. D. Castorina 110, 22460-320 Rio de Janeiro, Brazil}
\thanks{$^2$Instituto de Matem\'atica, Universidade Federal Fluminense, Niter\'oi, Brazil}
\email{viana\@@impa.br}\urladdr{www.impa.br/~viana/}
\email{yangjg\@@impa.br}

\begin{document}

\begin{abstract}
For a class of partially hyperbolic $C^k$, $k > 1$ diffeomorphisms with circle center
leaves we prove existence and finiteness of physical (or Sinai-Ruelle-Bowen)
measures, whose basins cover a full Lebesgue measure subset of the ambient manifold.
Our conditions contain an open and dense subset of all $C^k$ partially hyperbolic
skew-products on compact circle bundles.

Our arguments blend ideas from the theory of Gibbs states for diffeomorphisms with
mostly contracting center direction together with recent progress in the theory of
cocycles over hyperbolic systems that call into play geometric properties of invariant
foliations such as absolute continuity. Recent results show that absolute continuity
of the center foliation is often a rigid property among volume preserving systems.
We prove that this is not at all the case in the dissipative setting,
where absolute continuity can even be robust.
\end{abstract}

\maketitle
\setcounter{tocdepth}{1}
\tableofcontents

\section{Introduction}

Let $f:N\to N$ be a diffeomorphism on some compact Riemannian manifold $N$.
An invariant  probability $\mu$ is a \emph{physical (Sinai, Ruelle, Bowen) measure}
for $f$ if the set of  points $z\in N$ for which
\begin{equation}\label{eq.SRBmeasure}
\frac 1n \sum_{j=0}^{n-1} \delta_{f^i(z)} \to \mu
\quad \text{(in the weak$^*$ sense)}
\end{equation}
has positive volume. This set is denoted $B(\mu)$ and called the \emph{basin}
of $\mu$. A program for investigating the physical measures of partially hyperbolic
diffeomorphisms was initiated by Alves, Bonatti, Viana in \cite{ABV00,BoV00},
who proved existence and finiteness when $f$ is either ``mostly expanding''
(asymptotic forward expansion) or ``mostly contracting'' (asymptotic forward
contraction) along the center direction.

In this paper we analyze the existence and finiteness problem without a priori
conditions on the behavior along the center direction, in the case when the center
bundle has dimension $1$. Our results are illustrated by the following example.

Suppose $N=M\times S^1$, for some compact manifold $M$,
and $f_0:N\to N$ is a partially hyperbolic \emph{skew-product}
\begin{equation}\label{eq.skewproduct}
f_0:M\times S^1\to M\times S^1, \quad f_0(x,\theta)=(g_0(x),h_0(x,\theta))
\end{equation}
with center bundle $E^c$ coinciding with the vertical direction
$\{0\}\times T S^1$ at every point.
This implies $g_0$ is an Anosov diffeomorphism, and we also take it to be
transitive (all known Anosov diffeomorphisms being transitive).
Assume $f_0$ is of class $C^k$ for some $k>1$, not necessarily an integer.

\begin{main}\label{t.main1}
There exists a $C^k$ neighborhood $\cU_0$ of $f_0$ such that for every $f\in\cU_0$
which is accessible and whose center stable foliation is absolutely continuous
there exists a finite number of physical measures.
These measures are ergodic, the union of their basins has full volume in $N$,
and the center Lyapunov exponents are either negative or zero.
In the latter (zero) case the physical measure is unique.
\end{main}

The subset of accessible diffeomorphisms is $C^1$ open and $C^k$ dense in the
neighborhood of $f_0$ (Theorem 1.5 of Ni\c{t}ic\v{a}, T\"{o}r\"{o}k~\cite{NT01}).
Absolute continuity is also quite common in this context as we are going to see.
That is surprising, since Avila, Viana, Wilkinson~\cite{AVW1} have recently shown
that absolute continuity of the center foliation is a rigid property for
\emph{volume preserving} perturbations of skew-products. In contrast, here we prove

\begin{main}\label{t.main2}
Suppose $f_0$ exhibits some periodic vertical leaf $\ell$ such that $f_0^{per(\ell)} \mid \ell$
is Morse-Smale with a unique periodic attractor and repeller. Then $f_0$ is in the closure of an
open set $\cV$ of $C^k$ diffeomorphisms such that for every $f\in\cV$,
\begin{itemize}
\item the center stable, the center unstable, and the center foliation are absolutely continuous
\item both $f$ and its inverse have a unique physical measure, whose basin has full Lebesgue
measure in $N$.
\end{itemize}
\end{main}

Then the same is true for $f_0= g_0 \times \id$, since it is $C^k$ approximated by diffeomorphisms
as in the hypothesis of the theorem.

Although we are primarily interested in general (dissipative) diffeomorphisms, our methods also shed some
light on the issue of absolute continuity in the volume preserving context. Let $\lambda^c(f)$
denote the integrated center Lyapunov exponent of $f$ relative to the Lebesgue measure.

\begin{main}\label{t.main3}
For any small $C^1$ neighborhood $\cW$ of $f_0=g_0\times \id$ in the space of volume preserving
diffeomorphisms of $N$,
\begin{enumerate}
\item the subset $\cW_0$ of diffeomorphisms $f\in\cW$ such that $\lambda^c(f)\neq 0$ is $C^1$
open and dense in $\cW$; 
\item if $f\in\cW_0$ and $\lambda^c(f)>0$ then the center foliation and the center stable foliation
are not (even upper leafwise) absolutely continuous;
\item there exists a non-empty $C^1$ open set $\cW_1\subset\{f\in \cW_0: \lambda^c(f)>0\}$ such that
the center unstable foliation of every $g\in\cW_1$ is absolutely continuous.
\end{enumerate}
Claims (2) and (3) remain true when $\lambda^c(f)<0$, if one exchanges center stable with center unstable.
Every $C^k$, $k>1$ diffeomorphism $f\in\cW_1$ has a $C^k$ neighborhood $\cW_f$ in the space of all
(possibly dissipative) diffeomorphisms where the center unstable foliation remains absolutely continuous.
\end{main}

Theorems~\ref{t.main1} and~\ref{t.main2} follow from more detailed statements
that we present in the next section, where we also recall the main notions involved.
A discussion of the volume preserving case is given in Section~\ref{ss.conservative},
including the proof of Theorem~\ref{t.main3} and a (partly conjectural) scenario.

\section{Statement of results}\label{s.statementofresults}

Let $\cP_*^k(N)$ be the space of partially hyperbolic, dynamically coherent,
$C^k$ diffeomorphisms whose center leaves are compact, with any dimension,
and form a fiber bundle. Unless otherwise stated, we always assume $k>1$.
Most of our results concern the subspace $\cP_1^k(N)$ of diffeomorphisms with
$1$-dimensional center dimension. Let us begin by recalling the notions involved
in these definitions.

\subsection{Basic concepts}

A diffeomorphism $f:N \to N$ is \emph{partially hyperbolic} if there exists a continuous
$Df$-invariant splitting $TN=E^u \oplus E^c \oplus E^s$ and there exist constants $C>0$
and $\lambda <1$ such that
\begin{itemize}
\item[(a)] $\|Df_x^{-n}(v^u)\| \le C\lambda^n$ and
$\|Df_x^{n}(v^s)\| \le C\lambda^n$ \item[(b)] $\|Df_x^{-n}(v^u)\|
\le C\lambda^n \|Df_x^{-n}(v_c)\|$ and $\|Df_x^{n}(v^s)\| \le
C\lambda^n\|Df_x^{n}(v_c)\|$
\end{itemize}
for all unit vectors $v^u\in E_x^u$, $v^c\in E_x^c$, $v^s\in E_x^s$, and all $x\in N$ and
$n\ge 0$. Condition (a) means that the derivative $Df$ is uniformly expanding along $E^u$
and uniformly contracting along $E^s$.
Condition (b) means that the behavior of $Df$ along the \emph{center bundle} $E^c$ is
dominated by the behavior along the other two factors.
Here all three bundles are assumed to have positive dimension.

The bundles $E^u$ and $E^s$ are always integrable: there exist foliations $\cW^u$ and $\cW^s$
of $N$ tangent to $E^u$ and $E^s$, respectively, at every point. In fact these foliations
are unique. Moreover, they are \emph{absolutely continuous}, meaning that the projections
along the leaves between any two cross-sections preserve the class of sets with zero volume
inside the cross-section. See \cite{BP74,HPS77,Sh87}.
A diffeomorphism $f:N \to N$ is \emph{dynamically coherent} if the bundles $E^{cu}=E^c\oplus E^u$
and $E^{cs}=E^c\oplus E^s$ also admit integral foliations, $\cW^{cu}$ and $\cW^{cs}$.
Then, intersecting their leaves one obtains a \emph{center foliation} $\cW^c$ tangent at
every point to the center bundle $E^c$.

We say that the center leaves \emph{form a fiber bundle} over the leaf space $\Nc$ if for
any $\cW^c(x)\in N/\cW^c$ there is a neighborhood $V\subset N/\cW^c$ of $\cW^c(x)$ and a
homeomorphism
$$
h_x: V \times \cW^c(x) \to \pi_c^{-1}(V)
$$
smooth along the verticals $\{\ell\}\times\cW^c(x)$ and mapping each vertical onto the
corresponding center leaf $\ell$.

\begin{remark}
The fiber bundle condition is probably not necessary. Indeed, when the diffeomorphisms
are volume preserving, Avila, Viana, Wilkinson~\cite{AVW1} prove that if $\dim E^c=1$
and the generic center leaves are circles then the center leaves form a fiber bundle
\emph{up to a finite cover}. In particular, all leaves are circles.
Our arguments extend easily to this situation.
\end{remark}

A partially hyperbolic diffeomorphism $f:N\to N$ is \emph{accessible} if any points
$z$, $w\in N$ can be joined by a piecewise smooth curve $\gamma$ such that every
smooth leg of $\gamma$ is tangent to either $E^u$ or $E^s$ at every point.
Equivalently, every smooth leg of the curve $\gamma$ is contained in a leaf of
either $\cW^u$ or $\cW^s$.

The \emph{center Lyapunov exponent} $\lambda^c(\mu)$ of an $f$-invariant
probability measure $\mu$ is defined by
\begin{equation}\label{eq.Lyapunovexponent1}
\lambda^c(\mu) = \int \lambda^c(z) \, d\mu(z)
\quad \text{where $\lambda^c(z) = \lim_{n\to\infty} \frac{1}{n}
\log |Df^n\mid E^c_z|$.}
\end{equation}
By the ergodic theorem, this may be rewritten
\begin{equation}\label{eq.Lyapunovexponent2}
\lambda^c(\mu)
 = \int \log |Df \mid E^c_z| \, d\mu(z).
\end{equation}
If $\mu$ is ergodic then $\lambda^c(\mu)=\lambda^c(z)$ for $\mu$-almost every $z$.

Finally, the center direction is \emph{mostly contracting} (Bonatti, Viana~\cite{BoV00}) if
\begin{equation}\label{eq.mostlycontracting}
\limsup_{n\to+\infty} \frac 1n \log \|Df^n \mid E^c_x\| < 0.
\end{equation}
for a positive volume measure subset of any disk inside a strong unstable leaf.
It was shown by Andersson~\cite{An10} that this is a $C^k$, $k>1$ open property.

\subsection{The leaf space}\label{ss.leafspace}

Let $d$ be the Riemannian distance on $N$. We endow the leaf space $\Nc$ with the distance
defined by
$$
\dc(\xi,\eta)=\sup_{x\in\xi}\inf_{y\in\eta}d(x,y)+\sup_{y\in\eta}\inf_{x\in\xi}d(x,y)
\quad\text{for each } \xi, \eta \in \Nc.
$$
The quotient map $\pi_c:(N,d)\to(\Nc,\dc)$ is continuous and onto.
In particular, the metric space $(\Nc,\dc)$ is compact.

Let $\fc:\Nc\to\Nc$ be the map induced by $f$ on the quotient space $\Nc$.
The stable set of a point $\xi\in\Nc$ for $\fc$ is defined by
$$
W^s(\xi)=\{\eta\in\Nc : \dc(\fc^n(\xi),\fc^n(\eta))\to 0\text{ when }n\to +\infty\}
$$
and the local stable set of size $\vep>0$ is defined by
$$
W^s_\vep(\xi)=\{\eta\in\Nc : \dc(\fc^n(\xi),\fc^n(\eta))\le\vep
\text{ for all } n\ge 0\}.
$$
The unstable set and local unstable set of size $\vep>0$ are defined in the same way,
for backward iterates. It follows from the definitions that there exist constants
$K$, $\tau$, $\vep$, $\delta>0$ such that
\begin{enumerate}
\item $\dc(\fc^n(\eta_1),\fc^n(\eta_2))\le K e^{-\tau n}\dc(\eta_1,\eta_2)$ for
all $\eta_1, \eta_2 \in W^s_\vep(\xi)$, $n\ge 0$;
\item $\dc(\fc^{-n}(\zeta_1),\fc^{-n}(\zeta_2))\le Ke^{-\tau n}\dc(\zeta_1,\zeta_2)$ for
all $\zeta_1, \zeta_2 \in W^u_\vep(\xi)$, $n\ge 0$;
\item if $\dc(\xi_1,\xi_2)\le\delta$ then $W^s_\vep(\xi_1)$ and $W^u_\vep(\xi_2)$
intersect at exactly one point, denoted $[\xi_1\,\xi_2]$ and this point depends
continuously on $(\xi_1,\xi_2)$.
\end{enumerate}
This means that $\fc$ is a hyperbolic homeomorphism (in the sense of Viana~\cite{Almost}).
We denote $\cW^c(\Lambda)=\pi_c^{-1}(\Lambda)$, for any subset $\Lambda$ of $\Nc$.

By Anosov's closing lemma~\cite{An67}, periodic points are dense in the non-wandering set
of $\fc$. Smale's spectral decomposition theorem~\cite{Sm67},  the non-wandering set
splits into a finite number of compact, invariant, transitive, pairwise disjoint subsets.
Among these basic pieces of the non-wandering set, the \emph{attractors} $\Lambda_i$,
$i=1, \dots, k$ of $\fc$ are characterized by the fact that
$$
\Lambda_i = \bigcap_{n=0}^\infty \fc^n(U_i)
$$
for some neighborhood $U_i$ of $\Lambda_i$ and it is transitive.
The union of the stable sets $W^s(\Lambda_i)$, $i=1, \dots, k$ is an open dense subset of $\Nc$.
Every attractor $\Lambda_i$ consists of entire unstable sets, and so $\cW^c(\Lambda_i)$ is
$\cW^u$-saturated, that is, it consists of entire strong unstable leaves of $f$.
Additionally, every $\Lambda_i$ has finitely many connected components $\Lambda_{i,j}$,
$j=1, \dots, n_i$ that are mapped to one another cyclically. The unstable set $W^u(x)$ of
every $x\in \Lambda_{i,j}$ is contained and dense in $\Lambda_{i,j}$. In particular,
$\cW^c(\Lambda_{i,j})$ is also $\cW^u$-saturated. If $f_c$ is transitive, there is a unique
attractor $\Lambda_1=\Nc$.


We say $f$ is \emph{accessible on $\Lambda_i$} if, for every $j$, any points $z$, $w\in \cW^c(\Lambda_{i,j})$
can be joined by a piecewise smooth curve $\gamma$ such that every smooth leg of $\gamma$ is
tangent to either $E^u$ or $E^s$ at every point and the corner points belong to the same $\cW^c(\Lambda_{i,j})$.
The center direction of $f \mid {\cW^c(\Lambda_i)}$ is \emph{mostly contracting} if
\eqref{eq.mostlycontracting} holds for a positive volume measure subset of any disk inside a strong
unstable leaf contained in $\cW^c(\Lambda_i)$.

\subsection{Physical measures}\label{ss.finitenessandstability}

We are ready to state our main result on existence and finiteness of physical measures:

\begin{main}\label{t.mainB}
If $f\in\cP_1^k(N)$, $k > 1$ is accessible on every attractor and the center stable foliation
is absolutely continuous then, for each attractor $\Lambda_i$, either
\begin{itemize}
\item[(a)] there is a Lipschitz metric on each leaf of $\cW^c(\Lambda_i)$, depending continuously
on the leaf and invariant under $f$; then $f$ admits a unique physical measure, which is ergodic,
whose basin has full volume in the stable set of $\cW^c(\Lambda_i)$,
and whose center Lyapunov exponent vanishes;

\item[(b)] or the center direction of $f\mid {\cW^c(\Lambda_i)}$ is mostly contracting;
then $f\mid {\cW^c(\Lambda_i)}$ has finitely many physical measures, they are ergodic
for $f$ and Bernoulli for some iterate, the union of their basins is a full volume subset
of the stable set of $\cW^c(\Lambda_i)$, and their center Lyapunov exponents are negative.
\end{itemize}
The union of the basins of these physical measures has full volume in $N$.
\end{main}

To see that  Theorem~\ref{t.main1} is contained in Theorem~\ref{t.mainB} let us to note that,
for every $k\ge 1$, any $C^k$ partially hyperbolic skew-product $f_0$ is in the interior of
$\cP_1^k(N)$. Indeed, partial hyperbolicity is well known to be a $C^1$ open property and
the stability theorem for normally hyperbolic foliations (Hirsch, Pugh,  Shub~\cite{HPS77})
gives that every $f$ in a $C^1$ neighborhood of $f_0$ admits an invariant $\cW^{*}_f$
foliation, for each $*\in\{cu, cs, c\}$, and there exists a homeomorphism mapping the leaves
of $\cW^{*}_f$ diffeomorphically to the leaves of $\cW^{*}_{f_0}$. In particular, the center
leaves of $f$ form a circle fiber bundle.

\begin{remark}\label{r.continuoussection}
When the center fiber bundle is trivial, as happens near skew-products, part (a) of the
Theorem~\ref{t.mainB} gives that $f \mid {\cW^c(\Lambda_i)}$ is topologically conjugate
to a rotation extension
$$
\Lambda_i \times \RR/\ZZ \to \Lambda_i \times \RR/\ZZ,
\quad (x,\theta) \mapsto (\fc(x), \theta+\omega(x)).
$$
To see this, fix some consistent orientation of the center leaves and any continuous
section  $\sigma:N/\cW^c \to N$ of the center foliation, that is, any continuous map
such that $\sigma(\ell)\in\ell$ for every $\ell\in N/\cW^c$. Then define
$$
h: \cW^c(\Lambda_i) \to \Lambda_i \times \RR/\ZZ,
\quad h(z)=(\pi_c(z), |\sigma(\pi_c(z)),z|)
$$
where $|\sigma(\pi_c(z)),z|$ denotes the length, with respect to the $f$-invariant Lipschitz
metric, of the (oriented) curve segment from $\sigma(\pi_c(z))$ to $z$ inside the center leaf.
This map sends the center leaves of
$f$ to verticals $\{w\}\times \RR/\ZZ$, mapping the $f$-invariant Lipschitz metric on the center leaves
to the standard metric on $\RR/\ZZ$. Then $h \circ f \circ h^{-1}$ preserves the
standard metric measure on the verticals, and so it is a rotation extension, as stated.
Observe that, in addition, both $h$ and its inverse are Lipschitz on every leaf.
\end{remark}

Explicit bounds on the number of physical measures can be given in many cases.
For instance, we will see in Theorem~\ref{t.mainC} that if $f$ admits some periodic center leaf
$\ell$ restricted to which $f$ is Morse-Smale then the number of physical measures over the
attractor  containing $\pi_c(\ell)$ is bounded by the number of periodic orbits on $\ell$.
Notice that we must have alternative (b) of Theorem~\ref{t.mainB} in this case, since alternative (a)
is incompatible with the existence of hyperbolic periodic points.

We also want to analyze the dependence of the physical measures on the dynamics.
For this, we assume $N=M\times S^1$ and restrict ourselves to the subset $\cS^k(N)\subset\cP^k_1(N)$
of skew-product maps. We prove in Theorem~\ref{t.mainD} that there is an open and dense subset
of diffeomorphisms $f\in\cS^k(N)$ with mostly contracting center direction, such that the number of
physical measures is locally constant and the physical measures vary continuously with the diffeomorphism.
This property of \emph{statistical stability} has been studied in a number
of recent works, including Alves, Viana~\cite{AlV02}, V\'asquez~\cite{Va07},
Andersson~\cite{An10}.

As mentioned before, existence and finiteness of physical measures for partially
hyperbolic diffeomorphisms was proved by Alves, Bonatti, Viana~\cite{ABV00,BoV00},
under certain assumptions of weak hyperbolicity along the center direction.
Substantial improvements followed, by Alves, Luzzatto, Pinheiro~\cite{ALP03,ALP05},
Alves, Araujo~\cite{AA04}, Vasquez~\cite{Va07}, Pinheiro~\cite{Pi06}, and Andersson~\cite{An10},
among others. Perturbations of certain skew-products over hyperbolic maps have been studied
by Alves~\cite{Al00,Al03}, Buzzi, Sestier, Tsujii~\cite{BST}, and Gouezel~\cite{Gou06}.
In a remarkable recent paper, Tsujii~\cite{Tsu05} proved that generic (dense $G_\delta$)
partially hyperbolic surface \emph{endomorphisms} do admit finitely many physical measures,
such that the union of their basins has full Lebesgue measure. His approach is very different
from the one in the present paper and it is not clear how it could be extended to
diffeomorphisms in higher dimensions, even in the case of one-dimensional center bundle.

\subsection{Absolute continuity}\label{ss.absolutecontinuity}

It has been pointed out by Shub, Wilkinson~\cite{SW00} that foliations
tangent to the center subbundle $E^c$ are often \emph{not} absolutely
continuous. In fact, Ruelle, Wilkinson~\cite{RW01} showed that the disintegration
of Lebesgue measure along the leaves is often atomic.
Moreover, Avila, Viana, Wilkinson~\cite{AVW1} observed recently that for certain classes
of volume preserving diffeomorphisms, including perturbations of skew-products \eqref{eq.skewproduct},
absolute continuity of the center foliation is a rigid property: it implies that the center foliation
is actually smooth, and the map is smoothly conjugate to a rigid model.

However, we prove that this is not at all the case in our dissipative
setting:

\begin{main}\label{t.mainG}
There is an open set $\cU\subset \cP_1^k(N)$, $k>1$, such that the center
stable, the center unstable, and the center foliation are absolutely continuous
for every $f\in \cU$. Moreover, $\cU$ may be chosen to accumulate on every
skew-product map $f_0$ that admits a periodic vertical fiber restricted to
which the map is Morse-Smale with a unique periodic attractor and repeller.
\end{main}

Two weaker forms of absolute continuity are considered by Avila, Viana,
Wilkinson~\cite{AVW1}. Let $\Leb$ denote Lebesgue measure in the ambient manifold
and $\Leb_L$ be Lebesgue measure restricted to some submanifold $L$.
A foliation $\cF$ on $N$ is (\emph{lower}) \emph{leafwise absolutely continuous}
if for every zero $\Leb$-measure set $Y\subset N$ and $\Leb$-almost every $z\in M$,
the leaf $L$ through $z$ meets $Y$ in a zero $\Leb_L$-measure set.
Similarly, $\cF$ is \emph{upper leafwise absolutely continuous}
if $\Leb_L(Y)=0$ for every leaf $L$ through a full measure subset of points
$z\in M$ implies $\Leb(Y)=0$.
Absolute continuity implies both lower and upper leafwise absolute continuity
(see \cite{AVW1,BS02}); the converse is not true in general.
We will see in Proposition~\ref{p.mainF} that the center stable foliation of
a partially hyperbolic, dynamically coherent diffeomorphism with mostly
contracting center direction is always upper leafwise absolutely continuous.
This does not extend to lower leafwise absolutely continuity, in general:
robust counter-examples will appear in~\cite{VY2}; see also Example~\ref{ex.simpleexample}
for a related construction. However, as stated before, full absolute continuity
of the center foliation does hold on some open subsets of diffeomorphisms with
mostly contracting center.

\section{Gibbs $u$-states}\label{s.Gibbsustates}

Let $f:N\to N$ be a partially hyperbolic diffeomorphism. In what follows we denote
$I_r=[-r,r]$ for $r>0$ and $d_*=\dim E^{*}$ for each $*\in\{u, cu, c, cs, s\}$.
We use $\Leb^*$ to represent the volume measure induced by the restriction of the
Riemannian structure on the leaves of the foliation $\cW^*$ for each
$*\in\{u, cu, c, cs, s\}$.

Following Pesin, Sinai~\cite{PS82} and Alves, Bonatti, Viana~\cite{ABV00,BoV00}
(see also \cite[Chapter~11]{Beyond}), we call \emph{Gibbs $u$-state} any invariant
probability measure $m$ whose conditional probabilities (Rokhlin~\cite{Ro52})
along strong unstable leaves are absolutely continuous with respect to the volume
measure $\Leb^u$ on the leaf. More precisely, let
$$
\Phi:I_1^{d_u} \times I_1^{d_{cs}} \to N
$$
be any \emph{foliated box} for the strong unstable foliation.
By this we mean that $\Phi$ is a homeomorphism and maps every horizontal plaque
$I_1^{d_u}\times\{\eta\}$ diffeomorphically to a disk inside some strong unstable leaf.
Pulling $m$ back under $\Phi$ one obtains a measure $m_\Phi$ on $I_1^{d_u} \times I_1^{d_{cs}}$.
The definition of Gibbs $u$-state means that there exists a measurable function
$\alpha_\Phi(\cdot\,,\cdot)\ge 0$ and a measure $m_\Phi^{cs}$ on $I_1^{d_{cs}}$ such that
\begin{equation}\label{eq.Gibbsu}
m_\Phi(A) = \int_A \alpha_\Phi(\xi,\zeta) \,d\xi\,dm_\Phi^{cs}(\zeta)
\end{equation}
for every measurable set $A \subset I_1^{d_u}\times I_1^{d_{cs}}$.

Proofs for the following basic properties of Gibbs $u$-states can be found in
Section~11.2 of Bonatti, D\'\i az, Viana~\cite{Beyond}:

\begin{proposition}\label{p.Gibbsustates}
Let $f:N\to N$ be a partially hyperbolic diffeomorphism.
\begin{enumerate}

\item The densities of a Gibbs $u$-state with respect to Lebesgue measure along
strong unstable plaques are positive and bounded from zero and infinity.

\item The support of every Gibbs $u$-state is $\cW^u$-saturated, that is, it
consists of entire strong unstable leaves.

\item The set of Gibbs $u$-states is non-empty, weak$^*$ compact, and convex.
Ergodic components of Gibbs $u$-states are Gibbs $u$-states.

\item Every physical measure of $f$ is a Gibbs $u$-state and, conversely, every ergodic $u$-state
whose center Lyapunov exponents are negative is a physical measure.
\end{enumerate}
\end{proposition}

Now let $f\in\cP^k_*(N)$. Recall that $\pi_c:N\to\Nc$ denotes the natural quotient map
and $\fc:\Nc\to\Nc$ is the hyperbolic homeomorphism induced by $f$ in the leaf space.
Given small neighborhoods $V^s_\xi\subset W_\vep^s(\xi)$ and $V^u_\xi\subset W_\vep^u(\xi)$
inside the corresponding stable and unstable sets, the map
\begin{equation}\label{eq.bracket}
(\eta,\zeta) \mapsto [\eta,\zeta]
\end{equation}
defines a homeomorphism between $V^u_\xi\times V^s_\xi$ and some neighborhood
$V_\xi$ of $\xi$.
A probability measure $\mu$ on $\Nc$ has \emph{local product structure}
if for $\mu$-almost every point $\xi$ and any such product neighborhood $V_\xi$ the restriction
$\mu \mid V_\xi$ is equivalent to a product $\nu^u\times\nu^s$, where $\nu^u$ is a measure
on $V_\xi^u$ and $\nu^s$ is a measure on $V_\xi^s$.

%

In the sequel we prove three additional facts about Gibbs $u$-states that
are important for our arguments.

\begin{proposition}\label{p.finitenessinleafspace}
Take $f \in \cP_*^k(N)$, $k>1$ such that the center stable foliation
is absolutely continuous. For every ergodic Gibbs $u$-state $m$ the
support of the projection $(\pi_c)_*(m)$ coincides with some attractor
of $\fc$. In particular, periodic points are dense in the support of
$(\pi_c)_*(m)$.

Moreover, any two such projections with the same support must coincide.
In particular, the set of projections of all ergodic Gibbs $u$-states of
$f$ down to $\Nc$ is finite.
\end{proposition}

\begin{proposition}\label{p.ustates_product}
Take $f\in\cP_*^k(N)$, $k>1$ such that the center stable foliation
is absolutely continuous. If $m$ is a Gibbs $u$-state for $f$ then
$\mu=(\pi_c)_*(m)$ has local product structure.
\end{proposition}

\begin{remark}
Suppose $f$ is volume preserving. The Lebesgue measure $\Leb$ is both
an $s$-state and a $u$-state, because the strong stable foliation
and the strong unstable foliation are both absolutely continuous.
Thus, Proposition~\ref{p.ustates_product} implies that
$(\pi_c)_*(m)$ has local product structure if \emph{either}
$\cW^{cu}$ \emph{or} $\cW^{cs}$ is absolutely continuous.
\end{remark}

\begin{proposition}\label{p.somenegative}
Let $f\in \cP_*^k(N)$, $k>1$ and $\Lambda$ be an attractor of $\fc$.
Suppose the center stable foliation of $f$ is absolutely continuous
and $f$ is accessible on $\Lambda$.
Then every ergodic Gibbs $u$-state of $f$ supported in $\cW^c(\Lambda)$
has some non-positive center Lyapunov exponent.
\end{proposition}

As a special case, we get that if $f\in\cP_1^k(N)$, $k>1$ is accessible
on an attractor $\Lambda$ of $\fc$ and the center stable foliation is
absolutely continuous, then the (unique) center Lyapunov exponent of every
ergodic Gibbs $u$-state supported in $\cW^c(\Lambda)$ is non-positive.

The proofs of these propositions are given in Sections~\ref{ss.finiteness}
through~\ref{ss.allnegative}.

\subsection{Finiteness in leaf space}\label{ss.finiteness}

Here we prove Proposition~\ref{p.finitenessinleafspace}. Let $m_1$ be any ergodic
Gibbs $u$-state and $\mu_1=(\pi_c)_*(m_1)$. Notice that $\mu_1$ is ergodic and so
its support is a transitive set for $\fc$. Moreover, $\supp\mu_1=\pi_c(\supp m_1)$
consists of entire unstable sets, because the support of $m_1$ is $\cW^u$-saturated
(Proposition~\ref{p.Gibbsustates}). Thus, $\supp\mu_1$ is an attractor $\Lambda$
of $\fc$. As pointed out before, periodic points are dense in each attractor of $\fc$.

Now we only have to show that if $\mu_2=(\pi_c)_* m_2$ for another ergodic Gibbs
$u$-state $m_2$ and $\supp\mu_2=\Lambda=\supp\mu_1$ then $\mu_1=\mu_2$.
For this, take $x_c\in \Lambda$, let $U_c$ be a neighborhood of $x_c$ in the
quotient space $\Nc$, and let $U=\pi_c^{-1}(U_c)$. Then $U$ has positive $m_i$-measure
for $i=1, 2$. So, since the $m_i$ are ergodic Gibbs $u$-states, there are disks
$D_i\subset U$, $i=1, 2$ contained in strong unstable leaves and such that Lebesgue almost every
point in $D_i$ is in the basin $B(m_i)$ of $m_i$. Moreover, these disks may be chosen
such that the center stable foliation induces a holonomy map $h^{cs}: D_1 \to D_2$.
Since the center stable foliation is absolutely continuous, it follows that $h^{cs}$
maps some point $x_1\in D_1 \cap B(m_1)$ to a point $x_2\in D_2\cap B(m_2)$ in the
basin of $m_2$. Then $x_1$ and $x_2$ belong to the same center stable leaf of $f$,
and so their projections $\pi_c(x_1)$ and $\pi_c(x_2)$ belong to the same stable
set of $\fc$. Notice that $\pi_c(B(m_i))\subset B(\mu_i)$ for $i=1, 2$, and so each
point $\pi(x_i)\in B(\mu_i)$. Since either basin consists of entire stable sets,
this proves that $B(\mu_1)$ and $B(\mu_2)$ intersect each other, and so $\mu_1=\mu_2$.
This completes the proof of Proposition~\ref{p.finitenessinleafspace}.

\subsection{Local product structure}\label{ss.localproductstructure}

Here we prove Proposition~\ref{p.ustates_product}. Let $m$ be any Gibbs $u$-state and $\ell_0$
be any center leaf. Since the center leaves form a fiber bundle, we may find a neighborhood 
$V\subset N/\cW^c$ and a homeomorphism 
$$
\phi: V \times \ell_0 \mapsto \pi_c^{-1} (V), \quad (\ell,\zeta) \mapsto \phi(\theta,\zeta)
$$
that maps each vertical $\{\ell\}\times \ell_0$ to the corresponding center leaf $\ell$.
Clearly, we may choose $V$ to be the image of the bracket (recall Section~\ref{ss.leafspace})
$$
W^u_\vep(\ell_0)\times W^s_\vep(\ell_0) \to V, \quad (\xi,\eta) \mapsto [\xi,\eta]
$$
for some small $\vep>0$. Then, by dynamical coherence, the homeomorphism 
\begin{equation}\label{eq.recthomeo}
W^u_\vep(\ell_0)\times W^s_\vep(\ell_0) \times \ell_0 \to \pi_c^{-1}(V),
\quad (\xi,\eta,\zeta) \mapsto \phi([\xi,\eta],\zeta)
\end{equation}
maps each $\{\xi\}\times W_\vep^s(\ell_0) \times \ell_0$ onto a center stable leaf and
each $W_\vep^u(\ell_0) \times \{\eta\}\times \ell_0$ onto a center unstable leaf.
For each $x\in \pi_c^{-1}(V)$, let  $\cW^u_{loc}(x)$ denote the local strong unstable leaf over $V$,
that is, the connected component of  $\cW^u(x)\cap \pi_c^{-1}(V)$ that contains $x$.
Each $\cW^u_{loc}(x)$ is a graph over the unstable set $W^u(\pi_c(x))$ and the center stable
holonomy defines a homeomorphism
$$
h_{x,y}^{cs}: \cW^u_{loc}(x) \to \cW^u_{loc}(y)
$$
between any two local strong unstable leaves. By assumption, all these homeomorphisms are
absolutely continuous. Now let
$$
m \mid \pi_c^{-1}(V) = \int m_x \,d\hat m
$$
be the disintegration of $m$ relative to the partition of $\pi^{-1}_c(V)$ into local strong unstable leaves.
By definition of Gibbs $u$-states, each $m_x$ is equivalent to the Lebesgue measure along
$\cW^u_{loc}(x)$. It follows that the center stable holonomies are absolutely continuous relative to
the conditional probabilities of $m$ along local strong unstable leaves: 
\begin{equation}\label{eq.abscontrel1}
m_x(E) = 0 \text{ if and only if } m_y(h_{x.y}^{cs}(E))=0 
\end{equation} 
for $x$ and $y$ in some full $m$-measure subset of $\pi_c^{-1}(V)$ and for any measurable
set $E \subset \cW^u_{loc}(x)$. By the construction of \eqref{eq.recthomeo}, center stable
holonomies preserve the coordinate $\xi$. Thus, identifying $\pi_c^{-1}(V)$ with the space
$W_\vep(\ell_0)\times W^s_\vep(\ell_0) \times \ell_0$ through the homeomorphism 
\eqref{eq.recthomeo}, property \eqref{eq.abscontrel1} becomes
\begin{equation}\label{eq.abscontrel2}
m_x(A\times W_\vep^s(\ell_0)\times\ell_0) = 0 \text{ if and only if } 
m_y(A\times W^s_\vep(\ell_0)\times\ell_0)=0 
\end{equation} 
for any measurable set $A \subset W^u_\vep(\ell_0)$ and for $m$-almost every $x$ and $y$ in
$\pi_c^{-1}(V)$. Let $\mu \mid V = \int \mu^u_\eta \,d\mu^s(\eta)$ be the disintegration of $\mu$
relative to the partition of $V$ into unstable slices $W^u(\ell_0)\times\{\eta\}$; notice that $\mu^s$
is just the projection of $\mu\mid V$ to $W^s_\vep(\ell_0)$.
Projecting $m\mid \pi_c^{-1}(V)$ down to $V \approx W^u_\vep(\ell_0) \times W^s_\vep(\ell_0)$, 
property \eqref{eq.abscontrel2} yields
\begin{equation}\label{eq.abscontrel3}
\mu_\eta(A\times W_\vep^s(\ell_0)) = 0 \text{ if and only if }  \mu_{\eta'}(A\times W^s_\vep(\ell_0))=0 
\end{equation} 
for any measurable set $A \subset W^u_\vep(\ell_0)$ and for $\mu$-almost every $\eta$ and $\eta'$
in $V$. This means that the conditional probabilities $\mu^u_\eta$ are (almost) all equivalent.
Consequently, there is $\rho:W^u_\vep(\ell_0)\times W^s_\vep(\ell_0)\to (0,\infty)$ such that 
$\mu^u_\eta = \rho(\cdot,\eta) \mu^u$ at $\mu$-almost every point, where $\mu^u$ denotes the
projection of $\mu\mid V$ to $W^u_\vep(\ell_0)$. Replacing in the disintegration of $\mu\mid V$,
we get that $\mu \mid V = \rho \, \mu^u \times \mu^s$. This proves that $\mu$ has local product
structure, as claimed.

\subsection{Positive Gibbs $u$-states}\label{ss.allnegative}

Here we prove Proposition~\ref{p.somenegative}.
We begin by proving the following fact, which is interesting in itself:

\begin{proposition}\label{p.finiteGamma}
For $f\in \cP_*^1(N)$, given $c>0$ and $l\ge 1$ there
is $n_0$ such that $\#\big(S \cap \Gamma_{c,l}\big) < n_0$ for
every center leaf $S$, where
$$
\Gamma_{c,l}
 =\{x\in N:\liminf\frac{1}{n}\sum_{i=1}^n \log\|Df^{-l}\mid {E^c(f^{il}(x))}\|^{-1}\ge c\}.
$$
\end{proposition}

\begin{proof}
Recall that $\Leb^{c}$ denotes the Riemannian volume on center leaves.
The main ingredient is

\begin{lemma}\label{l.volume}
Given $c>0$ and $l \ge 1$ there exists $\delta>0$ such that for
any $x\in S \cap \Gamma_{c,l}$ and any neighborhood $U$ of $x$
inside the center leaf $S$ that contains $x$, one has
$$
\liminf\frac{1}{n}\sum_{i=0}^{n-1}\Leb^c(f^{i l}(U))\ge\delta.
$$
\end{lemma}

\begin{proof}
Let $x\in  S \cap \Gamma_{c,l}$ be fixed. Fix $0 < c_{1} <
c_{2} < c$ and define $H(c_{2})$ to be the set of $c_2$-hyperbolic times
for $x$, that is, the set of times $m\ge 1$ such that
\begin{equation}\label{eq.hyperbolictimes}
\frac{1}{k}\sum_{i=m-k+1}^{m} \log\|Df^{-l} \mid E^c_{f^{il}(x)}\|^{-1} \ge c_{2}
\quad\text{for all $1 \leq k \le m$.}
\end{equation}
By the Pliss Lemma (see \cite{Al00,ABV00}) , there exist $n_{1}\ge 1$ and $\delta_{1}>0$
such that
$$
\#\big(H(c_{2}) \cap [1,n)\big) \ge n \delta_{1}
\quad\text{for all $n\ge n_{1}$.}
$$
Notice that \eqref{eq.hyperbolictimes} implies $Df^{-k l}$ is an exponential
contraction on $E^{c}_{f^{ml}(x)}$:
$$
\|Df^{-k l} \mid E^c_{f^{ml}(x)}\|
\le \prod_{i=m-k+1}^{m} \|Df^{-l} \mid E^c_{f^{il}(x)}\|
\le e^{- c_{2} k}
\quad\text{for all $1 \le k \le m$.}
$$
It also follows from \cite{ABV00} that the points $f^{ml}(x)$ with $m\in H(c_{2})$
admit backward-contracting center disks with size uniformly bounded from below:
there is $r>0$ depending only on $f$ and the constants $c_{1}$ and $c_{2}$
such that
$$
f^{-kl}(B_r^c(f^{ml}(x))) \subset B^c_{e^{-c_1 k} r}(f^{(m-k)l}(x))
\quad\text{for all $1 \leq k \leq m$.}
$$
where $B^c_\rho(y)$ denotes the ball inside $\cW^{c}_{y}$ of radius $\rho$
around any point $y$. Let $a_{1}>0$ be a lower bound for $m^c(B^c_r(y))$
over all $y \in N$. Fix $n_2$ such that the ball of radius $e^{-c_1 k} r$ around
$x$ is contained in $U$ for every $k\ge n_2$. Then, in particular,
$$
f^{ml}(U) \supset B^{c}_{r}(f^{ml}(x))
\quad\text{and so}\quad m^{c}(f^{ml}(U))\ge a_{1}
$$
for every $m\in H(c_2)$ with $m\ge n_{2}$. So, for $n\gg\max\{n_1,n_2\}$,
$$
\frac{1}{n}\sum_{i=0}^{n-1}m^c(f^{i l}(U))
 \ge \frac 1n a_{1} \big[\#(H(c_{2}) \cap [1,n))-n_{2}\big]
 \ge \frac 1n a_{1} \big[n\delta_{1}-n_{2}\big]
 \ge \frac {\delta_{1}}{2} a_1\,.
$$
To finish the proof of Lemma~\ref{l.volume} it suffices to take $\delta=a_1\delta_{2}/2$.
\end{proof}

To deduce Proposition~\ref{p.finiteGamma} from Lemma~\ref{l.volume},
take any $n_0 \ge V/\delta$ where $V$ is an upper bound for the
volume of center leaves. Suppose $S \cap \Gamma_{c,l}$ contains
$n_0$ distinct points $x_j$, $j=1, \ldots, n_0$. Let $U_j$, $j=1,
\ldots, n_0$ be pairwise disjoint neighborhoods of the $x_j$
inside $S$. Take $n$ large enough that
$$
\frac{1}{n}\sum_{i=0}^{n-1}m^c(f^i(U_j))>\delta
\quad\text{for $1\leq j \leq n_0$.}
$$
Then
$$
V \ge \frac{1}{n}\sum_{i=0}^{n-1}m^c(f^i(S))
 \ge \sum_{j=1}^{n_0} \frac{1}{n} \sum_{i=0}^{n-1} m^c(f^i(U_j))
 > n_0\delta > V.
$$
This contradiction proves Proposition~\ref{p.finiteGamma}.
\end{proof}

\begin{proof}[Proof of Proposition~\ref{p.somenegative}]
We argue by contradiction. Suppose there exists some ergodic Gibbs
$u$-state $\nu$ supported in $\cW^c(\Lambda)$ whose center
Lyapunov exponents are all positive.

\begin{lemma}
There is $k_0 \ge 1$ and some ergodic Gibbs $u$-state $\nu_*$ of
$f^{k_0}$ supported in $\cW^c(\Lambda)$ such that
\begin{equation}\label{eq.nue}
\int \log \|Df^{-k_0} \mid E^c_x\|^{-1} d\nu_*(x) > 0.
\end{equation}
\end{lemma}

\begin{proof}
Arguing as in \cite[Section~2.1]{Almost} one can find $k_0\ge 1$
such that
$$
\int \log \|Df^{-k_0} \mid E^c_x\|^{-1} d\nu(x) > 0
$$
The measure $\nu$ needs not be ergodic for $f^{k_0}$ but, since it
is ergodic for $f$, it has a finite number $k$ of ergodic
components $\nu_i$ ($k$ divides $k_0$). Moreover,
$$
\int \log \|Df^{-k_0} \mid E^c_x\|^{-1} d\nu_i(x) > 0
$$
for some ergodic component $\nu_i$. Since, by Proposition~\ref{p.Gibbsustates},
each ergodic component $\nu_i$ is a Gibbs $u$-state, this completes the proof
of the lemma.
\end{proof}

Let $k_0\ge 1$ be fixed from now on and $\lambda>0$ denote the
expression on the left hand side of \eqref{eq.nue}. Let
$g=f^{k_0}$ and
$$
\Gamma = \{x\in N: \lim_{n\to\infty} \frac 1n \sum_{j=1}^{n} \log
\|Dg^{-1} \mid E^c_{g^{j}(x)}\|^{-1} = \lambda\}
$$
be the set of regular points of $\log \|Dg^{-1} \mid E^c\|$ for
the transformation $g$. By ergodicity, $\nu_*(\Gamma)=1$. A
statement similar to the next corollary was proved by Ruelle,
Wilkinson~\cite{RW01} when the diffeomorphism is $C^{1+\vep}$ and
the center is $1$-dimensional.

\begin{corollary}\label{c.Amie}
There is $n_0\ge 1$ such that $\#\big(\cW^c(w) \cap \Gamma\big) <
n_0$ for every $w\in N$.
\end{corollary}

\begin{proof}
Just use Proposition~\ref{p.finiteGamma} with $c=\lambda/2$ and
$l=k_0$. Clearly, $\Gamma\subset\Gamma_{c,l}$.
\end{proof}

Let $\ell_0$ be any periodic center leaf intersecting $\supp\nu_{*}$
(periodic center leaves are dense in the support, by Proposition~\ref{p.finitenessinleafspace})
and $\kappa \ge 1$ be minimal such that $g^{\kappa}(\ell_0)=\ell_0$.
Since $\nu_*$ is a Gibbs $u$-state and $\Gamma$ has full measure,
$\Leb^u(\cW^u(x)\setminus\Gamma)=0$ for $\nu_*$-almost every $x$,
where $\Leb^u$ denotes the Riemannian volume along strong unstable manifolds.
In particular, the stable set $\cW^s(\ell_0)=\cup_{z\in \ell_0}\cW^s(z)$ must intersect
some strong unstable disk $D^u$ such that $\Leb^u(D^u\setminus\Gamma)=0$.
See Figure~\ref{f.finite1}.

\begin{figure}[h]
\begin{center}
\psfrag{x}{$x$}\psfrag{y}{$y$}\psfrag{u}{$D^u$}
\psfrag{l}{$\ell_0$}\psfrag{cs}{$\cW^{s}(\ell_0)$}
\includegraphics[height=1.5 in]{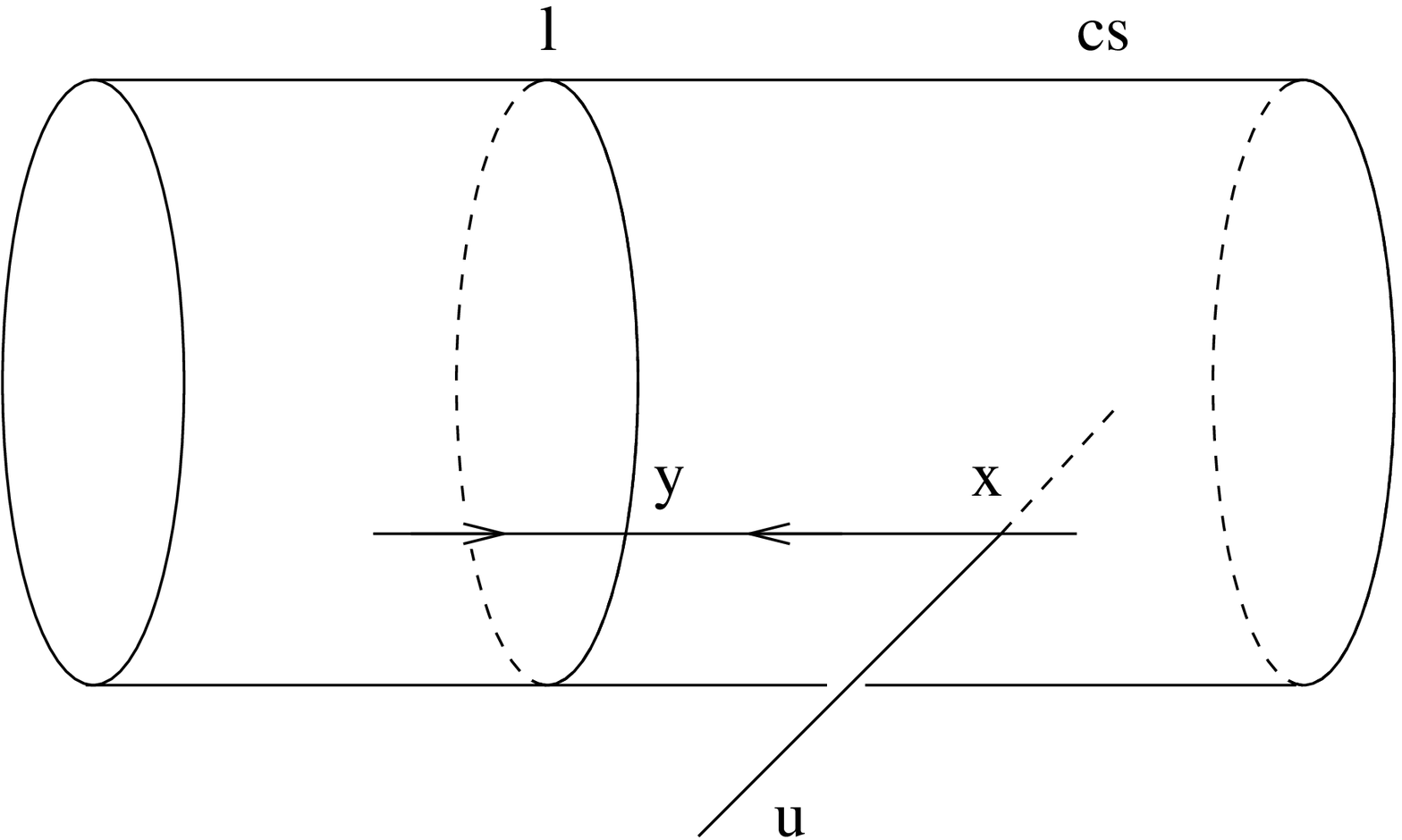}
\caption{\label{f.finite1}}
\end{center}
\end{figure}

\begin{lemma}\label{l.claim1}
Every point $x\in D^u \cap \cW^s(\ell_0)$ belongs to the strong
stable manifold of some periodic point $y\in \ell_0$ of $f$ with
period bounded by $k_0 \kappa n_0$.
\end{lemma}

\begin{proof}
Let $y\in \ell_0$ be such that $x\in\cW^s(y)$ and let
$g_{0}=g^\kappa\mid \ell_0$. Suppose first that the orbit of $y$
under $g_{0}$ is infinite. We refer the reader to Figure~\ref{f.finite2}.
Fix $y^*\in \omega(y)$ and let $(y_j)_j$
be an injective sequence of iterates of $y$ converging to $y^*$.
Let $(x_j)_j$ be a sequence of iterates of $x$ with $x_j\in\cW^s(y_j)$
and $d(x_j,y_j) \to 0$. Choose disks $D^u_j$
around the $x_j$ inside the forward iterates of $D^u$, small but
with \emph{uniform size}. Since $\Gamma$ is an invariant set,
$m^u(D_j^u\setminus\Gamma)=0$ for every $j$. For every large $j$,
the center leaves $\cW^c(x_j)$ is close to $\ell_0$ and so one can
define a $cs$-holonomy map $\pi^{cs}$ from $D^u_j$ to the local
strong unstable leaf through $y^*$. Since $\cW^{cs}$ is absolutely
continuous, the image of every $D_j^u\cap\Gamma$ is a full volume
measure subset of a neighborhood of $y^*$ inside $\cW^u(y^*)$, where
these neighborhoods also have uniform size for all large $j$. Let
$J=\{j_0, j_0+1, \ldots, j_0+n_0\}$ where $j_0$ is some large
integer and $n_0$ is as in Corollary~\ref{c.Amie}. On the one
hand, it follows from the previous considerations that
$$
\Gamma^*=\bigcap_{j\in J} \pi^{cs}(D_j^u\cap\Gamma)
$$
is a full volume measure subset of some neighborhood of $y^*$
inside $\cW^u(y^*)$. Fix some $w\in \Gamma^*$ close to $y^*$.
For each $j\in J$, let $w_j \in D_j^u \cap \Gamma$ be such that
$\pi^{cs}(w_j)=w$. Moreover, let $z_j$ be the point where the
local strong stable manifold of $w_j$ intersects
$\cW^{cu}(y^*)=\cW^{cu}(w)$. It is clear from the definition that
$w_j\in\cW^{cs}(w)$ and so $z_j\in\cW^c(w)$ for all $j \in J$.
Moreover, by choosing $w$ close enough to $y^*$ we can ensure that
$w_j$ is close to $x_j$ for every $j\in J$ and so $z_j$ is close
to $y_j$ for all $j\in J$. The latter implies that the $z_j$ are
all distinct. Observe also that $z_j\in\Gamma$ for all $j\in J$,
because $\Gamma$ is (clearly) saturated by strong stable leaves.
This proves that $\#(\cW^c(w) \cap \Gamma) \ge \# J
> n_0$, in contradiction with Corollary~\ref{c.Amie}. This
contradiction proves that the $g_{0}$-orbit of $y$ can not be
infinite.

\begin{figure}[h]
\begin{center}
\psfrag{xj}{$x_j$}\psfrag{yj}{$y_j$}\psfrag{u}{$D^u_j$}
\psfrag{cu}{$\cW^{cu}$}\psfrag{cs}{$\cW^{cs}$}\psfrag{w}{$w$}
\psfrag{zj}{$z_j$}\psfrag{wj}{$w_j$}\psfrag{y}{$y^*$}
\psfrag{uu}{$\cW^u$}
\includegraphics[height=1.8 in]{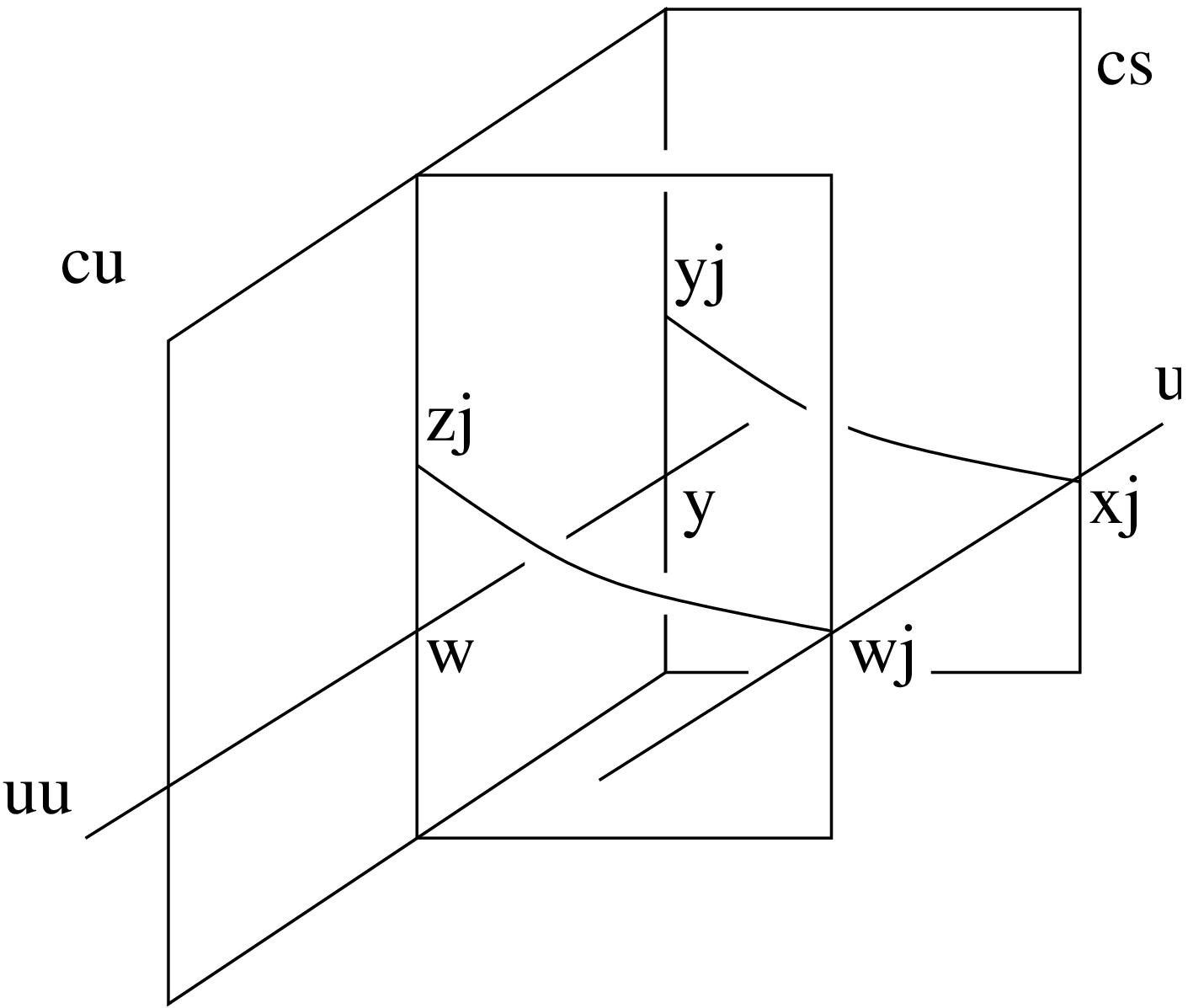}
\caption{\label{f.finite2}}
\end{center}
\end{figure}

Similar arguments handle the case when $y$ is a periodic point for
$g_0$. Let $k\ge 1$ be the (minimal) period of $y$ for $g_{0}$.
Forward iterates of $D^u$ accumulate on the strong unstable
manifolds of the iterates of $y$. Using, in much the same way as
before, that the center stable foliation is absolutely continuous
and $\Gamma$ is saturated by strong stable leaves, we find
$w\in\cW^{cu}(y)$ arbitrarily close to $y$ whose center leaf
$\cW^{c}_{w}$ intersects $\Gamma$ at points close to each of the
$k$ iterates of $y$. In view of Corollary~\ref{c.Amie} this
implies that $k < n_{0}$. This means that the period of $y$ for
$f$ is less than $k_0\kappa n_0$ as stated. The proof of
Lemma~\ref{l.claim1} is complete.
\end{proof}

\begin{lemma}\label{l.claim2}
Every point $z \in \ell_0$ is periodic for $f$, with period bounded
by $k_0\kappa n_0$.
\end{lemma}

\begin{proof}
Let $y\in \ell_0$ be a periodic point as in Lemma~\ref{l.claim1} and
let  $z\in \ell_0$ be arbitrary. Choose $y' \in\cW^u(y)\cap
\cW^c(\Lambda_i)\setminus \ell_0$ and $z' \in \cW^s(z)\cap
\cW^c(\Lambda_i)\setminus \ell_0$. By accessibility, there exists
some $su$-path connecting $y'$ to $z'$ or, in other words, there
exist points
$$
b_0 = y, a_1 = y', b_1, \ldots, a_i, b_i, \ldots, a_s= z', b_s=z
$$
which belong to $\cW^c(\Lambda_i)$ such that $a_j$ and $b_j$
belong to the same strong stable manifold and $b_j$ and $a_{j+1}$
belong to the same strong unstable manifold. We are going to find
an (arbitrarily) nearby $su$-path
\begin{equation}\label{eq.newsupath}
\tilde b_0=y, \tilde a_1, \tilde b_1, \ldots, a_i, b_i, \ldots, \tilde a_s, \tilde b_s
\end{equation}
with $\tilde b_s \in \ell_0$ and such that every $\tilde b_i$ belongs
to some periodic center leaf in $\cW^c(\Lambda_i)$. The first step
is to observe that, since periodic leaves are dense, one may
always find periodic leaves $\ell_1, \ldots, \ell_{s-1}$ arbitrarily
close to $\cW^c({b_1}), \ldots, \cW^c({b_{s-1}})$, respectively.
Let $\ell_s=\ell_0$. Assume $\tilde b_0, \tilde a_1, \ldots, \tilde b_k$
have been defined, for some $0 \le k < s$. Since $\cW^u(b_k)$
intersects the stable set of $\cW^c(b_{k+1})$ transversely at
$a_{k+1}$, and stable and unstable sets vary continuously with the
base point, we can find $\tilde b_{k+1}\in \ell_{k+1}$ close to
$b_{k+1}$ such that $\cW^u(\tilde b_k)$ intersects $\cW^s(\tilde
b_{k+1})$ at some point $\tilde a_{k+1}$ close to $a_k$. Repeating
this procedure $s$ times, we obtain an $su$-path as in
\eqref{eq.newsupath}.

The next step is to prove that the points $\tilde b_i$ themselves
are periodic. Recall that $\tilde b_0=y$ is taken to be periodic
and $D^u$ intersects $\cW^s(\tilde b_0)$. So, the iterates
accumulate on $\cW^u(\tilde b_0)$ and, in particular, on $\tilde
a_1$. This implies there exist points $w\in \ell_1$ arbitrarily close
to $\tilde b_1$ whose strong stable manifold intersects $f^n(D^u)$
for some $n$. Since $\Gamma$ has full volume inside every
$f^n(D^u)$, we may use Lemma~\ref{l.claim1} to conclude that $w$
is periodic, with period uniformly bounded. Consequently, $\tilde
b_1$ itself is periodic. It also follows that the iterates of
$D^u$ accumulate on $\cW^u(\tilde b_1)$. This means we may now
repeat the construction with $\tilde b_1$ in the place of $\tilde
b_0$ and conclude that $\tilde b_2$ is periodic. After $s$ steps
we conclude that $\tilde z=\tilde b_{s}$ is periodic. Since
$\tilde z$ is arbitrarily close to $z$, and all the periods are
bounded, we get that $z$ itself is periodic. This completes the
proof of the lemma.
\end{proof}

In particular, Lemma~\ref{l.claim2} implies that no periodic point
on the support of $\nu_*$ is hyperbolic. This is a contradiction
since, by a classical result of Katok~\cite{Ka80}, the support of
any hyperbolic measure contains hyperbolic periodic points. This
completes the proof of Proposition~\ref{p.somenegative}.
\end{proof}

\section{Mostly contracting center}\label{s.mostlycontractedcenter}

In this section we prove some useful facts about partially hyperbolic diffeomorphisms
with mostly contracting center direction. We call \emph{$\cW^u$-disk} any image of a
ball in $E^u$ embedded inside some strong unstable leaf.

\begin{lemma}\label{l.characterization}
The center direction of $f$ is mostly contracting if and only if the center Lyapunov
exponents of all ergodic Gibbs $u$-states are negative.

If $f\in \cP_1^k(N)$, $k>1$ and $\Lambda$ is an attractor of $\fc$, then
the center direction of $f\mid {\cW^c(\Lambda)}$ is mostly contracting if and only if
the center Lyapunov exponent is negative for every ergodic Gibbs $u$-state supported
in $\cW^c(\Lambda)$.
\end{lemma}

\begin{proof}
Bonatti, Viana~\cite{BoV00} show that if the center direction is mostly contracting then
the center exponents of every ergodic Gibbs $u$-state are negative.
To prove the converse, let $D$ be any disk inside a strong unstable leaf.
By \cite[Lemma~11.12]{Beyond} every Cesaro accumulation point of the iterates of Lebesgue
measure on $D$ is a Gibbs $u$-state. By \cite[Lemma~11.13]{Beyond} every ergodic component
of a Gibbs $u$-state is again a Gibbs $u$-state. This implies that the iterates $f^n(D)$
accumulate on the support of some ergodic Gibbs $u$-state $\nu$.
The hypothesis implies that $\nu$-almost every point has a Pesin (local) stable manifold
which is an embedded disk of dimension $d_{cs}$. Using also the absolute continuity of
the Pesin stable foliation (Pesin~\cite{Pe76}), we conclude that a positive Lebesgue measure
subset of points in some $f^n(D)$ belong to the union of these $d_s$-disks.
This implies that \eqref{eq.mostlycontracting} holds on a positive Lebesgue measure subset
of $D$, as we wanted to show.

The second part of the lemma follows from similar arguments.
\end{proof}

\subsection{Supports of Gibbs $u$-states}

\begin{lemma}\label{l.disjoint}
If the center direction of $f$ is mostly contracting then the supports
of the ergodic Gibbs $u$-states of $f$ are pairwise disjoint.
\end{lemma}

\begin{proof}
Let $m_1$ and $m_2$ be ergodic Gibbs $u$-states of $f$ and suppose $\supp m_1 \cap \supp m_2$
contains some point $z$. Let $D$ be any $\cW^u$-disk around $z$.
Then $D \subset \supp m_1 \cap \supp m_2$, since the supports are $\cW^u$-saturated
(Proposition~\ref{p.Gibbsustates}).
By Lemmas~11.12 and~11.13 in~\cite{Beyond}, every ergodic component $\nu$ of every Cesaro
accumulation point of the iterates of Lebesgue measure on $D$ is an ergodic Gibbs $u$-state.
Clearly, the support of $\nu$ is contained in $\supp  m_1 \cap \supp  m_2$.
By Pesin theory (see \cite{BoV00} for this particular setting) $\nu$-almost
every point has a local stable manifold which is an embedded $d_{cs}$-disk.
Recall (Proposition~\ref{p.Gibbsustates}) that the density of Gibbs $u$-states
along strong unstable leaves is positive and finite. Thus, we may find a $\cW^u$-disk
$D_\nu\subset \supp\nu$ such that every point $x$ in a full Lebesgue measure
subset $D_\nu^*$ has a Pesin stable manifold and belongs to the basin of $\nu$.
Moreover, $D_\nu$ is accumulated by $\cW^u$-disks $D_i\subset\supp m_1$ such
that Lebesgue almost every point is in the basin of $m_1$.
Assuming $D_i$ is close enough to $D_\nu$, it must intersect the union of
the local stable manifolds through the points of $D_\nu^*$ on some positive
Lebesgue measure subset $D_i^*$ (because the Pesin local stable lamination
is absolutely continuous~\cite{Pe76}). Then $D_i^*$ is contained in the
basin of $\nu$, and some full Lebesgue measure subset is contained in the
basin of $m_1$. That implies $m_1=\nu$. Analogously, $m_2=\nu$,
and so the ergodic Gibbs $u$-states $m_1$ and $m_2$ coincide.
That completes the proof of the lemma.
\end{proof}

\begin{remark}\label{r.uniqueustate}
It follows from Proposition~\ref{p.Gibbsustates} and Lemma~\ref{l.disjoint}
that if $f$ has mostly contracting center direction and minimal strong unstable
foliation then it has a unique Gibbs $u$-state.
This was first observed in \cite{BoV00}.
\end{remark}

\begin{proposition}\label{p.minimalufoliation}
Suppose the center direction of $f$ is mostly contracting, and let $m$
be an ergodic Gibbs $u$-state of $f$. Then the support of $m$ has a
finite number of connected components.
Moreover, each connected component $S$ is $\cW^u$-saturated and $\cW^u(x)$
is dense in $S$ for any $x\in S$.
\end{proposition}

\begin{proof}
Let $p$ be any periodic point in the support of $m$ with stable index equal to $d_{cs}$
(such periodic points do exist, by Katok~\cite{Ka80}) and let $\kappa$ be its period.
By Proposition~\ref{p.Gibbsustates}, the unstable manifold of every $f^j(p)$
is contained in $\supp m$. We claim that $\cup_{j=1}^{\kappa}\cW^u(f^j(p))$ is dense 
in $\supp m$. To see this, let $D$ be any disk inside $\cW^u(p)$.
Consider the forward iterates of Lebesgue measure on $D$. Using Lemmas~11.12 
and~Lemma~11.13 in \cite{Beyond}, one gets that any ergodic component of any Cesaro
accumulation point of these iterates is an ergodic Gibbs $u$-state $\nu$ supported 
inside the closure of $\cup_{j=1}^{\kappa}\cW^u(f^j(p))$. By Lemma~\ref{l.disjoint},
the Gibbs $u$-states $m$ and $\nu$ must coincide. In particular, $\supp m$ is contained
in the closure of $\cup_{j=1}^{\kappa}\cW^u(f^j(p))$. That proves our claim.

Since $m$ is ergodic for $f$, its ergodic decomposition relative to $f^\kappa$ has the
form $m = l^{-1} \sum_{i=1}^l f^i_*\tilde m$ where $l$ divides $\kappa$ and $\tilde m$
is $f^\kappa$-invariant and ergodic. Then
$$
\supp m=\bigcup_{i=1}^{l}f^i(\supp\tilde m).
$$
We claim that the $f^i(\supp\tilde m)$, $i=1, \dots, l$ are precisely the connected
components of $\supp m$. On the one hand, the previous paragraph gives that
$p\in f^s(\supp\tilde m)$ for some $s$. Replacing either $p$ or $\tilde m$ by an iterate,
we may suppose $s=0$.
Then, by the argument in the previous paragraph applied to $f^\kappa$ (it is clear
from the definition \eqref{eq.mostlycontracting} that if $f$ has mostly contracting
then so does any positive iterate), $\supp\tilde m$ coincides with the closure of
$\cW^u(p)$ and, in particular, it is connected.
On the other hand, Lemma~\ref{l.disjoint} gives that the $f^i(\supp\tilde m)$,
$i=1, \dots, l$ are pairwise disjoint. Since they are closed, it follows that
they are also open in $\supp m$. This proves our claim.

We are left to prove that the strong unstable foliation is minimal in each
connected component $S_i=f^i(\supp\tilde m)$.
This will follow from an argument of Bonatti, D\'{i}az, Ures~\cite{BDU02}:

\begin{lemma}
There is a neighborhood $U^s_i$ of $f^i(p)$ inside $W^s(f^i(p))$ such that every
unstable leaf in $S_i$ has some transverse intersection with $U^s_i$.
\end{lemma}

\begin{proof}
For any $x\in S_i$, let $D_x$ be a small $\cW^u$-disk around $x$.
Since $\tilde m_{j}$ is the unique ergodic $u$-state of $f^{\kappa}$ with
support contained in $S_{j}$. It is also the unique Cesaro accumulation point of the
iterates of $\Leb_{D_x}$ under $f^{\kappa}$. In particular, there is $n_x\ge 1$ such
that $f^{n_x\kappa}(D_x)$ intersects the local stable manifold of $f^i(p)$
transversely. This implies that $D_x$ intersects the global stable manifold of
$f^i(p)$ transversely. Then, by continuity of the strong unstable foliation,
there is a neighborhood $V_x$ of $x$ and a bounded open set
$U_x\subset W_{f^\kappa}^s(f^i(p))$ such that $\cW^u(y)$ intersects $U_x$
transversely for every $y\in V_x$.
The family $\{V_x: x\in S_i\}$ is an open cover of the compact set $S_i$.
Let $\{V_{x_1},\cdots,V_{x_m}\}$ be a finite subcover. Choose $U^s_j$ a bounded
neighborhood of $f^i(p)$ inside $W_{f^\kappa}^s(f^i(p))$ containing $U_{x_j}$
for all $j=1, \dots, m$. It follows from the construction that every strong
unstable leaf contained in $S_i$ intersects $U_i^s$ transversely.
This finishes the proof of the lemma.
\end{proof}

Let us go back to proving Proposition~\ref{p.minimalufoliation}. The lemma
gives that $\cW^u(f^{-n\kappa}(x))$ intersects $U_i^s$ transversely, and so
$\cW^u(x)$ intersects $f^{n\kappa}(U_i^s)$ transversely, for every $x\in S_i$
and every $n\ge 0$. Since $f^{n\kappa}(U_i^s)$ converges to $f^i(p)$ when
$n\to\infty$, it follows that $W^u(f^i(p))$ is contained in the closure
of $W^u(x)$. Hence, $W^u(x)$ is dense in $S_j$, as claimed.
The proof of the proposition is complete.
\end{proof}

\subsection{Bernoulli property}\label{ss.bernoulliproperty}

An invariant ergodic measure $\eta$ of a transformation $g$ is called
\emph{Bernoulli} if $(g,\eta)$ is ergodically conjugate to a Bernoulli shift.

\begin{theorem}\label{t.mainI}
Suppose $f$ is a $C^k$, $k>1$ partially hyperbolic diffeomorphism with
mostly contracting center direction. Then there is $l\ge 1$ and a
$C^k$ neighborhood $\cal U$ of $f$ such that for any $g\in \cal U$,
every ergodic $u$-state of $g^l$ is Bernoulli.
\end{theorem}

\begin{proof}
Let $m_1, \dots,  m_u$ be the ergodic Gibbs $u$-states of $f$.
Proposition~\ref{p.minimalufoliation} gives that for each $j=1, \dots, u$ there exists
$l_j \ge 1$ such that the support of $m_j$  has $l_j$ connected components $S_{j,i}$,
$i=1, \dots, l_j$. Moreover, each connected component $S_{j,i}$ carries an ergodic
component $m_{j,i}=f^i_*\tilde m_j$ of the Gibbs $u$-state $m_j$ for the iterate $f^{l_j}$.
Let $l$ be any common multiple of $l_1, \dots, l_u$. Then every $S_{j,i}$ is
fixed under $f^l$. Moreover, every Gibbs $u$-state $m_{j,i}$ is $f^l$-invariant and
$f^{n l}$-ergodic for every $n\ge 1$: otherwise $S_i$ would break into more than one
connected component (cf. the proof of Lemma~\ref{l.disjoint}).
Then, by Ornstein, Weiss~\cite{OW98}, every $m_{i,j}$ is a Bernoulli
measure for $f^l$. We claim that $\{m_{j,i}: 1 \leq j \leq u \text{ and } 1 \leq i \leq l_j\}$
contains all the ergodic $u$-states of $f^{n l}$ for every $n \ge 1$. Indeed, let $m_*$ be any
ergodic $u$-state for $f^{n l}$. Then
$$
m=\frac{1}{nl} \sum_{k=1}^{nl} f_*^k m_*
$$
is a $u$-state for $f$. Let $m=a_1 m_1+\cdots+a_u m_u$ be its ergodic decomposition for $f$
and let $s$ be such that $a_s>0$. Then $\supp m_s \subset \supp m$.
Since $\supp m_s$ is $f$-invariant, it must intersect $\supp m_*$.
Using Lemma~\ref{l.disjoint} for $f^{nl}$ we conclude that $m_*$ must coincide with some
ergodic component of $m_s$ for the iterate $f^{nl}$. In other words, it must coincide with
$m_{s,i}$ for some $i=1, \dots, l_s$, and this proves our claim.

Now we extend these conclusions to any diffeomorphism $g$ in a $C^k$, $k>1$ neighborhood of $f$.
By Andersson~\cite{An10}, any such $g$ has mostly contracting center direction, and so the
previous argument applies to it. However, we must also prove that the integer $l$ can be
taken uniform on a whole neighborhood of $f$. Notice that the only constraint on $l$ was
that it should be a multiple of the periods $l_j$ of the ergodic components $m_j$.
Observe that \cite{An10} also gives that the number of ergodic Gibbs $u$-states does not
exceed the number of ergodic Gibbs $u$-states of $f$. So, we only need to check that the
periods $l_j$ remain uniformly bounded for any $g$ in a neighborhood.
We do this by arguing with periodic points, as follows. Let us fix, once and for all,
$f$-periodic points $p_j$ with stable index $d_{cs}$ in the support of each $m_j$,
$j=1, \dots, u$. The period of each $p_j$ is a (fixed) multiple of $l_j$.
Let $p_j(g)$ be the continuation of these periodic points for some nearby diffeomorphism $g$,
and let $\{m_1(g), \dots, m_s(g)\}$, with $s \le u$ be the ergodic Gibbs $u$-states
of $g$. We claim that every $\supp m_j(g)$, $1\leq j \leq s$ contains some $p_i(g)$,
$1\leq i \leq u$. This can be seen as follows. As observed before, any accumulation point
of Gibbs $u$-states of $g$ when $g\to f$ is a Gibbs $u$-state for $f$. We fix some small
$\vep>0$ and consider the $\vep$-neighborhoods $B(p_j,\vep)$ of the periodic points $p_j$.
Then, for any $g$ close enough to $f$ every ergodic Gibbs $u$-state $m_j(g)$ must give
positive weight to some $B(p_i,\vep)$ and, consequently, also to $B(p_i(g),2\vep)$.
By continuous dependence of stable manifolds of periodic points on the dynamics,
and the fact that the supports of Gibbs $u$-states are $u$-saturated, it follows that
$\supp m_j(g)$ contains some $\cW^u$-disk that intersects $W^s(p_i(g))$ transversely.
Then, the support of $m_j(g)$ must contain $p_i(g)$. This proves our claim.
It follows that the period $l_j(g)$ of each ergodic Gibbs $u$-state of $g$ divides
the period of some $p_i(g)$ which, of course, coincides with the period of $p_i$.
Since the latter have been fixed once and for all, this proves that the $l_j(g)$
are indeed uniformly bounded on a neighborhood of $f$. The proof of the theorem is
complete.
\end{proof}

\subsection{Abundance of mostly contracting center}

We also give a family of new examples of diffeomorphisms with mostly contracting center.

\begin{theorem}\label{t.mainJ}
Suppose $\dim M =3$. The set of ergodic diffeomorphisms such that either $f$ or $f^{-1}$
has mostly contracting center direction is $C^1$ open and dense in the space of $C^k$,
$k>1$ partially hyperbolic volume preserving diffeomorphisms with 1-dimensional center
and some fixed compact center leaf.
\end{theorem}

\begin{proof}
Denote by $\cPO_m^k$ the set of $C^k$ volume preserving partially hyperbolic diffeomorphisms
with 1-dimensional center and some fixed compact center leaf.
This is a $C^1$ open set, cf. \cite[Theorem~4.1]{HPS77}.
Moreover, the diffeomorphisms such that both the strong stable foliation and the strong
unstable foliation is minimal fill an open and dense subset $U_1$ of $\cPO_m^1$.
This follows from a conservative version of the results of \cite{BDU02}: one only has to
observe that blenders, that they use for the proof in the dissipative context, can be
constructed also in the conservative setting, as shown by \cite{HHTU10}.
By \cite{BB03}, there is an open and dense subset $U_2$ for which the center
Lyapunov exponent
$$
\int \log|Df\mid {E^c(x)}|dm(x) \neq 0.
$$
Furthermore, by \cite{DW03}, there is an open and dense subset $U_3$ of $\cPO^1_m$
consisting of accessible diffeomorphisms. Let $U=U_1 \cap U_2 \cap U_3$.
Before proceeding, let us recall that $C^\infty$ are $C^1$ dense in the space of
volume preserving diffeomorphisms, by \cite{Av09-dens}. In particular, the $C^1$ open
and dense subset $U$ has non-trivial intersection with the space of $C^k$
diffeomorphisms, for any $k>1$.

We claim that for every $C^k$, $k>1$ diffeomorphism $f$ in $U$,
either $f$ or its inverse has a unique ergodic Gibbs $u$-state and the
corresponding center Lyapunov exponent is negative.
In particular, by Lemma~\ref{l.characterization}, either $f$ or its inverse has mostly
contracting center direction. The first step is to note that $f$ is ergodic, since it is
accessible (see~\cite{BW08a,HHU07p,PS97}). Then the Lebesgue measure $\Leb$ is an ergodic
Gibbs $u$-state for both $f$ and $f^{-1}$. Since the strong stable and strong unstable
foliations are minimal, the Gibbs $u$-state is unique; see Remark~\ref{r.uniqueustate}.
This completes the proof of Theorem~\ref{t.mainJ}.
\end{proof}

%
%

\section{Finiteness and stability of physical measures}\label{s.finitenessandstability}

In this section we prove Theorem~\ref{t.mainB}. As remarked before, Theorem~\ref{t.main1}
is a particular case.
We begin by recalling certain ideas from Bonatti, Gomez-Mont, Viana~\cite{BGV03} and
Avila, Viana~\cite{AV3} that we use for handling the case when the
center Lyapunov exponent vanishes.

\subsection{Smooth cocycles}\label{ss.smoothcocycles}

By assumption, the center leaves of $f$ define a
fiber bundle $\pi_c:N\to\Nc$ over the leaf space. Then $f$ may be
seen as a smooth cocycle (as defined in \cite{AV3}) over $\fc$:
$$
\begin{array}{rrcl}
f : & N & \rightarrow & N \\
 & \downarrow & & \downarrow \\
\fc: & \Nc & \rightarrow & \Nc
\end{array}
$$
It follows from the form of our maps that the strong stable manifold $\cW^s(x)$ of every
point $x\in M$ is a graph over the stable set $W^s_{\pi_c(x)}$ of $\pi_c(x)\in\Nc$.
For each $\eta \in W^s(\xi)$, the strong stable holonomy defines a homeomorphism
$h^s_{\xi,\eta}:\xi \to \eta$ between the two center leaves.
In fact (see~\cite[Proposition~4.1]{AV3}),
\begin{equation}\label{eq.sholonomy}
h^s_{\xi,\eta}(\theta) = \lim_{n\to\infty} (f^n\mid\eta)^{-1} \circ (f^n \mid \xi)(\theta),
\end{equation}
(for large $n$ one can identify $f^n(\xi)\approx f^n(\eta)$ via the fiber bundle
structure) for each $\theta\in\xi$ and the limit is uniform on the set of all
$(\xi,\eta,\theta)$ with $\theta\in\xi$ and $\xi$ and $\eta$ in the same local stable set.
These \emph{$s$-holonomy} maps satisfy
\begin{itemize}
\item $h^s_{\eta,\zeta} \circ h^s_{\xi,\eta} = h^s_{\xi,\zeta}$ and $h^s_{\xi,\xi}=\id$
\item $f \circ h^s_{\xi,\eta} = h^s_{\fc(\xi),\fc(\eta)} \circ f$
\item $(\xi,\eta,\theta) \mapsto h^s_{\xi,\eta}(\theta)$ is continuous on the set of triples
$(\xi,\eta,\theta)$ with $\xi$ and $\eta$ in the same local stable set and
$\theta\in \cW^c(\xi)$.
\end{itemize}


Let $m$ be any $f$-invariant probability measure and $\mu=(\pi_c)_*(m)$.
A \emph{disintegration} of $m$ into conditional probabilities along the center leaves is a
measurable family $\{m_\xi:\xi\in\supp\mu\}$ of probability measures with $m_{\xi}(\xi)=1$
for $\mu$-almost every $\xi$ and
\begin{equation}\label{eq.dis}
m(E) = \int m_\xi(E)\,d\mu(\xi)
\end{equation}
for every measurable set $E\subset M$. By Rokhlin~\cite{Ro52}, such a family exists
and is essentially unique. A disintegration is called \emph{$s$-invariant} if
$$
(h^s_{\xi,\eta})_* m_\xi = m_\eta \quad\text{for every $\xi$, $\eta \in \supp\mu$ in the same stable set.}
$$
In a dual way one defines \emph{$u$-holonomy} maps and \emph{$u$-invariance}.
We call a disintegration \emph{bi-invariant} if it is both $s$-invariant and $u$-invariant,
and we call it \emph{continuous} if $m_\xi$ varies continuously with $\xi$ on the support of $\mu$,
relative to the weak$^*$ topology.

\begin{proposition}\label{p.continuousdisintegration}
Let $f\in \cP_*^k(N)$, $k>1$ be such that the center stable foliation is absolutely continuous.
Let $m$ be an ergodic Gibbs $u$-state with vanishing center Lyapunov exponents.
Then $m$ admits a disintegration $\{m_\xi:\xi\in\supp\mu\}$ into conditional probabilities along
the center leaves which is continuous and bi-invariant.
\end{proposition}

\begin{proof}
Proposition~\ref{p.ustates_product} gives that $(\pi_c)_*m$ has local product structure.
Thus, we are in a position to use Theorem~D of Avila, Viana~\cite{AV3} to obtain the
conclusion of the present proposition.
\end{proof}

\subsection{Zero Lyapunov exponent case}\label{ss.zero}

The following result provides a characterization of the systems exhibiting ergodic Gibbs
$u$-states with vanishing central exponent.

\begin{proposition}\label{p.elliptic}
Let $f\in \cP_1^k(N)$, $k>1$ be such that the center stable foliation is absolutely continuous.
Let $\Lambda$ be an attractor of $\fc$ such that $f$ is accessible on $\Lambda$,
and let $m$ be an ergodic Gibbs $u$-state with vanishing center Lyapunov exponent. Then
\begin{itemize}

\item[(1)] the conditional probabilities $\{m_x: x\in \Lambda\}$ along the center leaves are
equivalent to the Lebesgue measure on the leaves, with densities uniformly bounded from zero
and infinity;


\item[(2)] $\supp m = W^c(\Lambda)$ and this is the unique Gibbs $u$-state supported in $\cW^c(\Lambda)$;

\item[(3)] the basin of $m$ covers a full Lebesgue measure subset of a neighborhood
of $\cW^c(\Lambda)$.
\end{itemize}\end{proposition}

\begin{proof}
By Proposition~\ref{p.continuousdisintegration}, there is a disintegration
$\{m_x: x\in \Lambda\}$ of $m$ along the center foliation which is continuous,
$s$-invariant, and $u$-invariant. Let $\xi$ and $\eta$ be any two points in
$\cW^c(\Lambda)$. By accessibility on $\Lambda$, one can find an $su$-path
$
b_0 = \xi, b_1, \ldots, b_{s-1}, b_s=\eta
$
connecting $\xi$ to $\eta$. This $su$-path induces a holonomy map
$h:\cW^c(\xi)\to\cW^c(\eta)$, defined as the composition of all
strong stable/unstable holonomy maps $h_i: \cW^c(b_{i-1})\to\cW^c(b_{i})$.
The fact that the disintegration is bi-invariant gives, in particular,
that
\begin{equation}\label{eq.inva1}
m_\eta(h(B^c_\vep(\xi)) = m_\xi(B^c_\vep(\xi))
\end{equation}
It is a classical fact that the strong stable and strong unstable foliations are
absolutely continuous in a strong sense: their holonomy maps have bounded
Jacobians. See \cite{BP74,Man88,AbV}. Those arguments extend directly to
their restrictions to each center stable or center unstable leaf, respectively:
the restricted strong stable and strong unstable foliations are also absolutely
continuous with bounded Jacobians.
By compactness, the $su$-path may be chosen such that the number $s$ of legs and
the length of each leg are uniformly bounded, independent of $\xi$ and $\eta$
\footnote{This may be deduced from \cite{ASV} as follows. By Proposition~8.3 in \cite{ASV},
given any $x_0\in M$ there exists $w\in M$ such that $x_0$ is connected to every point
in a neighborhood of $w$ by a uniformly bounded $su$-path. Then the same is true if
one replaces $w$ by an arbitrary point $z\in M$: connect $w$ to $z$ by some $su$-path;
the "same" $su$-path determines a bijection between neighborhoods of $w$ and $z$;
concatenating with $su$-paths from $x_0$ to the neighborhood of $w$ one obtains
uniformly bounded $su$-paths from $x_0$ to any point near $z$. The claim now follows
by compactness of the ambient manifold.}.
Then, we may fix a uniform upper bound constant $K>1$ on the Jacobians of all associated
strong stable and strong unstable holonomies. Notice $\Leb^c(B_r(\zeta))=2r$,
since the center leaves are one-dimensional. Then
\begin{equation}\label{eq.inva2}
 K^{-1}  \Leb^c(B^c_{\vep}(\xi))
 \le \Leb^c(h(B^c_\vep(\xi)))
 \le K  \Leb^c(B^c_{\vep}(\xi)).
\end{equation}
From \eqref{eq.inva1} and \eqref{eq.inva2} we obtain
\begin{equation*}
\frac 1K \frac{m_\xi(B^c_\vep(\xi))}{\Leb^c(B^c_\vep(\xi))}
\le \frac{m_\eta(h(B^c_\vep(\xi)))}{\Leb^c(h(B^c_\vep(\xi)))}
\le K \frac{m_\xi(B^c_\vep(\xi))}{\Leb^c(B^c_\vep(\xi))}.
\end{equation*}
and, taking the limit as $\vep\to 0$,
\begin{equation*}
\frac 1K \frac{dm_\xi}{d\Leb^c}(\xi)
 \le \frac{dm_\eta}{d\Leb^c}(\eta)
\le K\frac{dm_\xi}{d\Leb^c}(\xi).
\end{equation*}
Since we can always find $\eta$ where the density is less
or equal than $1$ (respectively, greater or equal than $1$),
this implies that
\begin{equation}\label{eq.inva4}
\frac{dm_\xi}{d\Leb^c}(\xi) \in \big[K^{-1}, K]
\end{equation}
for every $\xi$, and that proves claim (1).


Now let $m'$ be any other ergodic Gibbs $u$-state supported in $\cW^c(\Lambda)$.
The center Lyapunov exponent of $m'$ must vanish: otherwise, by \cite{Ka80},
there would be some hyperbolic periodic point in $\cW^c(\Lambda)$, and that is
incompatible with the conclusion in part (1) that there exist invariant
conditional probabilities equivalent to Lebesgue measure along the center leaves.
So, all the previous considerations apply to $m'$ as well. In particular,
it has a continuous disintegration $\{m'_x: x\in \Lambda\}$ along the center
foliation such that each $m'_x$ is equivalent with $\Leb^c$. Moreover, by
Proposition~\ref{p.finitenessinleafspace}, $(\pi_c)_*(m)=(\pi_c)_*(m')$.
Then, $\Leb^c$-almost every point in almost every center leaf, relative to
$(\pi_c)_*(m)=(\pi_c)_*(m')$, belongs to the basin of both $m$ and $m'$.
In particular, the two basins intersect, and that implies $m=m'$.
That completes the proof of claim (2).

From conclusion (1) we get that there exists a full $m$-measure set
$\Delta$ consisting of leaves such that $\Leb^c$-almost every point
in the leaf belongs to the basin of $m$. Then, since $m$ is a Gibbs
$u$-state, we can find a  $\cW^u$-disk $D^u_0$ such that $\Delta\cap D^u_0$
has full measure in $D^u_0$. Consider the $cu$-disk
$$
D_0^{cu}=\bigcup_{\xi\in D_0^u}\cW^c(\xi)
$$
Observe that the center foliation $\cW^c$ is absolutely continuous on each
center unstable leaf, because the corresponding holonomy maps between unstable
leaves coincide with the corresponding holonomy maps for the center stable
foliation, which we assume to be absolutely continuous. Using this fact,
and a Fubini argument, we get that Lebesgue almost every point in $D_0^{cu}$
belongs to the basin of $m$. Next, consider the open set
$$
D_0=\bigcup_{\zeta\in D_0^{cu}} \cW^s_{loc} (\zeta).
$$
Since the strong stable foliation is absolutely continuous and the basin
is $\cW^s$-saturated, it follows that Lebesgue almost every point in $D_0$
belongs to the basin of $m$. It is clear that we can cover $\cW^c(\Lambda)$
by finitely many such open sets $D_0$. This proves (3), and so the proof of
the lemma is complete.
\end{proof}

\subsection{Construction of physical measures}

We are nearly done with the proof of Theorem~\ref{t.mainB}.
By Proposition~\ref{p.somenegative}, all ergodic Gibbs $u$-states have non-negative
center Lyapunov exponent. The case when the exponent vanishes for some Gibbs $u$-state
is handled by Proposition~\ref{p.elliptic}: we get alternative (a) of the theorem in
this case. Finally, if the center Lyapunov exponent is negative for all Gibbs
$u$-states over some attractor $\Lambda_i$ of $f_c$ then, by Lemma~\ref{l.characterization},
the center direction of $f$ is mostly contracting on that attractor $\Lambda_i$.
Then, by Bonatti, Viana~\cite{BoV00}, there are finitely many ergodic Gibbs $u$-states
supported in $\cW^c(\Lambda_i)$, these $u$-states are the physical measures of $f$,
and the union of their basins covers a full volume measure subset of a neighborhood
of $\cW^c(\Lambda_i)$. By Theorem~\ref{t.mainI}, all these physical measures are Bernoulli
for some iterate of $f$. Thus, we get alternative (b) of the theorem in this case.

From now on, let $\{ m_{i,j}\}_{j=1}^{J(i)}$ be the physical measures supported on each
attractor $\Lambda_i$. As we have just seen, their basins cover a full Lebesgue measure
subset of a neighborhood $U_i$ of $\cW^c(\Lambda_i)$. We want to prove that the union
of all these basins contains a full Lebesgue measure subset of the ambient manifold.
Suppose otherwise, that is, suppose the complement $C$ of this union has positive Lebesgue
measure. Let $C_0\subset C$ be the set of Lebesgue density points of $C$. Notice that $C_0$
is $f$-invariant and $\Leb(C_0)=\Leb(C)$. Since the unstable foliation is absolutely continuous,
there is a $\cW^u$-disk $D^u$ such that $\Leb_{D^u}(D^u\cap C_0)>0$. Denote $I^u=D^u\cap C_0$.
Then every Cesaro accumulation point of the iterates of Lebesgue measure on $I^u$ is a
Gibbs $u$-state (see \cite{Beyond}, section 11.2), and so its ergodic components are ergodic
Gibbs $u$-states. Let $m^*$ be any such accumulation point and $m_{i,j}$ be an ergodic
component of $m^*$. The support of $m_{i,j}$ is contained in $U_i$, and so there is $n_0 \ge 1$
such that $f^{n_0}(I^u)$ intersects $U_i$. Recalling that $C_0$ is invariant, we get that
$\Leb(C_0\cap U_i)>0$. This contradicts the definition of $C_0$, since Lebesgue almost every
point in $U_i$ belongs to the basin of $m_{i,l}$ for some $l=1, \dots, J(i)$.
This contradiction proves that the union of the basins does have full Lebesgue measure in $N$.
That completes the proof of Theorem~\ref{t.mainB}.

\subsection{Number of physical measures}

In this section, we give explicit upper bounds on the number of physical
measures for some diffeomorphisms with mostly contracting center direction:

\begin{theorem}\label{t.mainC}
Let $f\in\cP_1^k(N)$, $k>1$ be accessible on some attractor $\Lambda$ and have
absolutely continuous center stable foliation.
Assume there exists some center leaf $\ell\subset \cW^c(\Lambda)$ such that
$f^\kappa(\ell)=\ell$ for some $\kappa\ge 1$ and $f^{\kappa}\mid \ell$ is
Morse-Smale with periodic points $p_1, \cdots, p_s$.

Then the center direction is mostly contracting over $\Lambda$ and $f$ has
at most $s$ physical measures supported in $\cW^c(\Lambda)$.
If $\cW^u(p_i)$ intersects $W^s(\ell) \setminus \cup_{j=1}^s \cW^s(p_j)$
for every $i$ then $f$ has at most $s/2$ physical measures supported in
$\cW^c(\Lambda)$.
\end{theorem}

\begin{proof}
Since $f$ has hyperbolic periodic point in $\cW^c(\Lambda)$ the restriction
of $f$ to $\cW^c(\Lambda)$ can not be conjugate to a rotation extension over
$\Lambda$. Thus, by Theorem~\ref{t.mainB}, $f$ has mostly contracting center
direction over $\Lambda$.

\begin{lemma}\label{l.nouniformcontracting}
Suppose $f\in\cP_1^k(N)$, $k>1$ has mostly contracting center direction on an
attractor $\Lambda$ and let $p$ be any periodic point in $\cW^c(\Lambda)$.
Then any disk $D^u$ in unstable manifold of $p$ contains a positive measure
subset $I^u$ such that any $\xi\in I^u$ belongs to the basin of some physical
measure and has local stable manifold $W_{loc}^s(\xi)$.
\end{lemma}

\begin{proof}
As in the proof of Lemma~\ref{l.characterization}, there is a positive measure
subset $I^u$ of $D^u$ belonging to the basin of some physical measure $m$,
and for $\xi\in I^u$, there is $n_0$ such that $f^{n_0}(\xi)$ belongs to the
Pesin stable manifold of some point $\zeta$. Iterating backward we obtain a
local stable manifold for $\xi$.
\end{proof}

Suppose $f$ has physical measures $\{m_j\}_{j=1}^{J}$ on $\cW^c(\Lambda)$.
Let $p_t$, $t=1, \dots, s$ be fixed as in Theorem~\ref{t.mainC}.
Since the support of each physical measure is a $u$-saturated compact set,
the following fact is an immediate consequence of Lemma~\ref{l.disjoint}:

\begin{corollary}\label{c.uniqueinsupport}
For each $1\le t \le s$ there is at most one physical measure whose support
intersects $W^s(p_t)$.
\end{corollary}

%

As observed before, the unstable foliation is minimal in every attractor
in the quotient. So, the orbit of every strong unstable leaf intersects
$W^s(\ell)=\cup_{t=1}^s W^s(p_t)$. Since the supports of physical measures are
$\cW^u$-saturated and invariant, it follows that for every $1\leq j\leq J$
there exists some $1\leq t \leq s$ such that $\supp m_j$ intersects $W^s(p_t)$.
So, by Corollary~\ref{c.uniqueinsupport}, $J\leq s$.

Let $\{p_{s_i}\}_{i=1}^{s/2}$ be periodic points in $\ell$ with stable index $d_s$
(i.e. repellers for $f\mid\ell$) and let $\{p_{s_i}\}_{i=s/2+1}^s$ be periodic points
in $\ell$ with stable index $d_{cs}$ (i.e. attractors for $f\mid\ell$).
We claim that if $\cW^u(p_i)$ intersects $\cW^s(\ell) \setminus \cup_{j=1}^s \cW^s(p_j)$
for every $i$, then the support of every physical measure contains some $p_{i}$,
$s/2+1\leq i \leq s$. Indeed, by the previous observations the support must intersect
$W^s(p_i)$ for some $i$, corresponding to either an attractor or a repeller of
$f\mid\ell$. In the former case, the claim is proved; in the latter case, our assumption
on $\ell$ implies that the support intersects the stable set of some other periodic
point $p_j$ which is an attractor, and so the claim follows in just the same way.
So, by the previous argument, the number of physical measures can not exceed $s/2$
in this case. The proof of Theorem~\ref{t.mainC} is complete.
\end{proof}

\subsection{Statistical stability}\label{ss.statisticalstability}

We also want to analyze the dependence of the physical measures on the dynamics. For this,
we assume $N=M\times S^1$ and restrict ourselves to the set $\cS^k(N)\subset\cP_1^k(N)$
of skew-product maps. Notice that every $f\in\cS^k(N)$ is dynamically coherent, has
compact one-dimensional center leaves, and absolutely continuous center stable foliation.
As pointed out before, partially hyperbolicity is an open property and accessibility
holds on an open and dense subset of $\cS^k(N)$.

\begin{theorem}\label{t.mainD}
For any $k > 1$ there exists a $C^1$ open and $C^k$ dense subset $\cBk(N)$ of $\cS^k(N)$
such that every $f\in\cBk(N)$ has mostly contracting center direction.
Moreover, on a $C^k$ open and dense subset of $\cBk(N)$ the number of physical measures
is locally constant and the physical measures depend continuously on the diffeomorphisms.
\end{theorem}

\begin{proof}
Notice that every $f\in\cS^k(N)$ is dynamically coherent, has compact one-dimensional
center leaves, and absolutely continuous center stable foliation. By a variation of an
argument of Ni\c{t}ic\u{a}, T\"{o}r\"{o}k~\cite{NT01}, one gets that the set of diffeomorphisms
in $\cS^k(N)$ which are accessible on all attractors are $C^1$ open and $C^r$ dense.
Let us comment a bit on this, since our setting is not exactly the same. The heart of
the proof is to show that the accessibility class of any point contains the corresponding
center leaf. This is done by considering 4-leg $su$-paths linking the point to every
nearby point in the center leaf; in this way one gets that every accessibility class is
open in the center leaf; then, connectivity gives the conclusion. The only difference in
our case is that we deal with accessibility over each of the attractors, not the whole
ambient manifold. However, the arguments remains unchanged, just with the additional caution
to choose the corners of the $4$-leg $su$-path to be points over the attractor.
It is easy to see that the set of diffeomorphisms in $\cS^k(N)$
which have a center leaf containing some hyperbolic periodic point is $C^1$ open and $C^r$
dense. Take $\cBk$ be the intersection of above two sets. Then by Theorem~\ref{t.mainB},
any $f\in \cB^k$ has mostly contracting center bundle.
By Andersson~\cite{An10}, for any partially hyperbolic diffeomorphism $f$ with mostly
contracting center direction there is a $C^k$, $k>1$ neighborhood $\mathcal U$ of $f$
such that any $g\in \cal U$ has mostly contracting center direction also,
and on a $C^k$ open and dense subset of $\mathcal U$, the number of physical measures
is locally constant and the physical measures depend continuously on the diffeomorphism.
This ends the proof of Theorem~\ref{t.mainD}.
\end{proof}

\section{Absolute continuity for mostly contracting center}\label{s.criteria}

Throughout this section $f:N\to N$ is a partially hyperbolic, dynamically coherent, $C^k$, $k>1$
diffeomorphism with mostly contracting center direction. Recall the later is a robust (open) condition,
by Andersson~\cite{An10}. We develop certain criteria for proving absolute continuity of the
center stable, center unstable, and center foliations and we apply these tools to exhibit several
robust examples of absolute continuity. In particular, this yields a proof of Theorem~\ref{t.mainG}.

The starting point for our criteria is the observation that for maps with mostly contracting
center the Pesin stable manifolds are contained in, and have the same dimension as the center
stable leaves. Since the Pesin stable lamination is absolutely continuous (\cite{Pe76,PS89}),
in this way one can get a local property of absolute continuity for the center stable foliation.
This initial step of the construction is carried out in Section~\ref{ss.localabscont}.
Then one would like to propagate this behavior to the whole ambient manifold, in order to obtain
actual absolute continuity. It is important to point out that this can not possibly work without
additional conditions. Example~\ref{ex.simpleexample} below illustrates some issues one encounters.
A more detailed analysis, including explicit robust counter-examples will appear in \cite{VY2}.
Suitable assumptions are introduced in Section~\ref{ss.criteria}, where we also give the
precise statements of our criteria. In Section~\ref{ss.recurrence} we present the main tool
for propagating local to global behavior. The criteria are proved in Sections~\ref{ss.upperleafwise}
through~\ref{ss.newcriterion}.

Before proceeding, let us give a simple example of a map whose center foliation is leafwise absolutely
continuous and locally absolutely continuous, but not globally absolutely continuous.
This kind of construction explains why Pesin theory alone can not give (global) absolute continuity
of center foliations, even when the center direction is mostly contracting.

\begin{example}\label{ex.simpleexample}
Let us start with $f_0:S^1\times [0,1] \to S^1\times [0,1]$, $f_0(x,t)=(2x,g(t))$ where $g:[0,1]\to[0,1]$
is a $C^2$ diffeomorphism such that $g(0)=0$, $g(1)=1$, $g(t)<t$ for all $t\in (0,1)$, and $0 < g'(t)< 2$
for every $t\in[0,1]$. Then $f_0$ is a partially hyperbolic endomorphism of the cylinder, with the vertical
segments as center leaves. Next, let $f:S^1\times [0,1] \to S^1\times [0,1]$ be a $C^2$-small perturbation,
preserving the two boundary circles $\cC_i = S^1\times\{i\}$, $i=0, 1$ and the vertical line $\{0\}\times[0,1]$
through the fixed point $(0,0)$. Moreover, the horizontal derivatives of $f$ at the endpoints of this vertical
line should be different:
\begin{equation}\label{eq.derivativesaredifferent}
\frac{\partial f}{\partial x}(0,0) \neq \frac{\partial f}{\partial x}(0,1).
\end{equation}
By the stability of center foliations (\cite{HPS77}, the new map $f$ has a center foliation whose
leaves are curve segments with endpoints in the two boundary circles. Thus, they induce a holonomy map
$h:\cC_0\to\cC_1$ that conjugates the two expanding maps $f \mid \cC_0$ and $f \mid \cC_1$.
Condition \eqref{eq.derivativesaredifferent} implies that the conjugacy can not be absolutely continuous
(see \cite{SS85}). This shows that the center foliation is not absolutely continuous. Yet, it is absolutely
continuous restricted to $S^1 \times [0,1)$, as we are going to explain. Notice that our assumptions
imply that $g'(0) < 1 < g'(1)$ and so the lower boundary component $\cC_0$ is an attractor for $f_0$,
with $S^1 \times [0,1)$ as its basin of attraction. Then the same is true for the perturbation $f$.
Moreover, restricted to this basin, the center leaves coincide with the Pesin stable manifolds of
the points in the attractor, and so they do form an absolutely continuous foliation. In particular,
this also shows that the center foliation is leafwise absolutely continuous.
\end{example}

\subsection{Criteria for absolute continuity}\label{ss.criteria}

We assume that some small cone field around the strong unstable bundle has been fixed.
We call \emph{$u$-disk} any embedded disk of dimension $d_u$ whose tangent space is
contained in that unstable cone field at every point. Previously, we introduced the
special case of $\cW^u$-disks, which are contained in strong unstable leaves.
To begin with, in Section~\ref{ss.upperleafwise} we prove that upper leafwise
absolute continuity always holds in the present context:

\begin{proposition}\label{p.mainF}
The center stable foliation of $f$ is upper leafwise absolutely continuous,
if it exists.
\end{proposition}

For the next criterion we assume the diffeomorphism is non-expanding along
the center direction. This notion is defined as follows. Assume also $f$
is dynamically coherent. Given $r>0$ and $*\in\{s,cs,c,cu,u\}$, we denote
by $\cW^*_r(x)\subset\cW^*(x)$ the ball of radius $r$ around $x$,
relative to the distance induced by the Riemannian metric of $N$ on the
leaf $\cW^*(x)$. In what follows we always suppose $r$ is small enough
so that $\cW^*_r(x)$ is an embedded disk of dimension $d_*$ for all $x\in M$
and every choice of $*$. We use $\hW^s(p)$ and $\hW^u(p)$ to denote the
stable and unstable sets of a periodic point $p$. We say that $f$ is
\emph{non-expanding along the center direction} if there exist $\rho>0$
and $\vep>0$ such that
\begin{itemize}
\item $f^n(\cW^{cs}_\vep(x)) \subset\cW^{cs}_\rho(f^n(x))$ for every
$n\ge 0$ and almost any $x$ in any $u$-disk.
\item the support of every ergodic Gibbs $u$-state $m$ contains some
periodic point $p$ such that $\hW^s(p) \supset \cW^{cs}_{2\rho}(p)$
\end{itemize}

\begin{proposition}\label{p.nonexpansion}
If $f$ is non-expanding along the center direction then the center
stable foliation is absolutely continuous.
\end{proposition}

The proof of this proposition is given in Sections~\ref{ss.localabscont} through~\ref{ss.nonexpansion}.
We will see that the hypothesis holds for a classical construction of partially hyperbolic,
robustly transitive diffeomorphisms due to Ma\~n\'e~\cite{Man78} (Section~\ref{ss.Mane}).
It also holds for a more recent class of examples introduced by Bonatti, Viana~\cite{BoV00},
which are not even partially hyperbolic (though they do admit a dominated invariant splitting
of the tangent bundle), but this fact will not be proved here.

Let $f\in\cP_1^k(N)$. Let $\ell$ be a periodic center leaf $\ell$, with period $\kappa\ge 1$.
For $*\in\{s,u\}$, we denote $\cW^*(\ell)=\cup_{\zeta\in\ell}\cW^*(\zeta)$.
We call \emph{homoclinic leaf} associated to $\ell$ any center leaf $\ell'$ contained in
$\cW^s(\ell)\cap\cW^u(\ell)$. Then there exist strong stable and strong unstable holonomy maps
\begin{equation}\label{eq.usholonomies}
h^s:\ell\to\ell' \quand h^u:\ell\to\ell'
\end{equation}
We say that $\ell$ is \emph{in general position} if
\begin{itemize}
\item[(a)] $f^\kappa \mid \ell$ is Morse-Smale with a single periodic
attractor $a$ and a single periodic repeller $r$;

\item[(b)] $h^s(a \cup r)$ is disjoint from $h^u(a \cup r)$, for some
homoclinic leaf associated to the center leaf $\ell$.

\end{itemize}
Notice that $\cW^s(\ell')\setminus\cW^s(h^s(r))$ is contained in the
stable manifold $\hW^s(a)$ of the attractor. Thus, condition (b)
implies that $\cW^u(a)$ and $\cW^u(r)$ intersect $\hW^s(a)$ transversely.
Analogously, $\cW^s(a)$ and $\cW^s(r)$ intersect $\hW^u(r)$ transversely.

\begin{proposition}\label{p.generalposition}
Suppose $f\in\cP^k_1(N)$ has some center leaf $\ell$ in general position
and such that every strong unstable leaf intersects $\cW^s(\ell)$.
Then the center stable foliation of $f$ is absolutely continuous.
\end{proposition}

This proposition is proved in Section~\ref{ss.newcriterion}.
In Section~\ref{ss.proofofmainG} we use it to prove Theorems~\ref{t.main2}
and~\ref{t.mainG}, and in Section~\ref{ss.conservative} we give an
application to volume preserving systems.
Noticing that, apart from dynamical coherence, all the hypotheses of
Proposition~\ref{p.generalposition} are robust, we get the following
immediate consequence:

\begin{corollary}\label{c.generalposition}
Suppose $f\in\cP_1^k(N)$ is robustly dynamically coherent and has some periodic
center leaf $\ell$ in general position and such that every strong unstable leaf
intersects $\cW^s(\ell)$. Then the center stable foliation is robustly
absolutely continuous.
\end{corollary}

\subsection{Local absolute continuity}\label{ss.localabscont}

The following lemma will allow us to obtain some property of local absolute continuity:

\begin{lemma}\label{l.csbox}
For any ergodic $u$-state $m$ of $f$ and any disk $D$ contained in an
unstable leaf inside $\supp m$, there is a positive measure subset
$\Gamma$ such that the points in $\Gamma$ have (Pesin) stable manifolds
with uniform size. Moreover, these stable manifolds form an absolutely
continuous lamination, in the following sense: there is $K>0$ such that
for any two $u$-disks $D_1$, $D_2$ sufficiently close to $D$, the stable
manifolds of points in $\Gamma$ define a holonomy map between subsets of
$D_1$ and $D_2$, and this is absolutely continuous, with Jacobian between
$1/K$ and $K$.
\end{lemma}

\begin{proof}
Because $f$ has mostly contracting center direction, $m$ is a hyperbolic
ergodic measure of $f$, by Pesin theory, there is a Pesin block $\Lambda$
with positive $m$ measure such that every point $x\in \Lambda$ has uniform
size of stable manifold, and these stable manifolds on $\Lambda$ is
uniformly absolutely continuous. Notice that the stable manifolds are contained
in the center stable leaves. Since $m$ is a $u$-state, there is a disk $D_0$
contained in an unstable leaf inside the support and intersecting $\Lambda$
on a  $m^u$-positive measure subset $D_0^*$. Then the points in $D_0^*$ have
stable manifolds of size bounded below by some $\delta_0>0$.
Denote $B_0=\cup_{x\in D_0^*}W^s_{\delta_0}(x)$. Since $m$ is a $u$-state,
$m(B_0)=a_0>0$. We claim that there is $n_0>0$ such that $(f^{n_0})_*\Leb_D(B_0)\neq 0$.

Let us prove this claim. Let $D_\vep^*$ be the $\vep$-neighborhood of $D_0^*$
inside the corresponding unstable leaf. Denote by
$B_\varepsilon=\cup_{x\in D_\varepsilon^*}W^{cs}_{\delta_0}(x)$,
it is an open set, and $m(B_\varepsilon)\geq a_0>0$.
Because every Cesaro accumulation point of the iterates of
Lebesgue measure on $D$ is a Gibbs $u$-state with support contained
in $\supp m$, and there is a unique ergodic $u$-state with
support contained in $\supp m$, then $m$ is the unique Cesaro
accumulation of the iterates of Lebesgue measure on $D$.
Since $B_\varepsilon$ is open, one has $\lim_{n\rightarrow
\infty}\frac{1}{n}\sum_{i=0}^{n-1}(f^i)_*\Leb_{D}(B_\varepsilon)\geq
m(B_\varepsilon)=a_0$, so there is arbitrarily big $n$ such that
$f_*^{n}(\Leb_D)(B_\varepsilon)>a_0/2$. For $\delta>0$ sufficiently
small, denote by $D_\delta=\{x\in D, d^u(x,\partial (D))\geq
\delta\}$, one has $m^u(D\setminus D_\delta)<a_0/4$. Then there is $y\in
D_\delta$ such that $f^n(y)\in B_\varepsilon$ and $f^n(D)$
contains a disk $D_y$ around $y$ and for any $x\in D_\varepsilon^*$
one has $W^{cs}_{\delta_0}(x)\cap D_y\neq \emptyset$.
Then the stable manifolds of $D_0^*$ define a holonomy map
between $D_0^*$ and $B_0\cap D_y$, by the uniform absolute
continuity of these stable manifolds, $\Leb_{D_y}(D_y\cap B_0)>0$,
then $f_*^{n_0}(\Leb_D)(B_0)>0$. This proves the claim.

This claim implies $\Leb_{f^n(D)}(f^n(D)\cap B^*_0)>0$, let
$\Gamma =D\cap f^{-n_0}(B^*_0)$, then $\Leb_D(\Gamma)>0$, every
point in $\Gamma$ has uniform size of stable manifold, and these
stable manifolds are uniformly absolutely continuous.
\end{proof}

Suppose $f\in \Diff^k(N)$, $k>1$ admits dominated splitting
$E^{cu}\oplus E^{cs}$, and it is dynamically coherent, that is,
it has center stable and center unstable foliation.
We call \emph{$cs$-block} for $f$ the image $\cB=h(\Sigma \times I^{d_{cs}})$
of any embedding $h: \Sigma\times I^{d_{cs}} \to N$, with
$\Sigma\subset I^{d_{cu}}$, satisfying the following properties:
\begin{enumerate}
\item $h(\{a\}\times I^{d_{cs}})$ is contained in $\cW^{cs}(h(a, 0))$, for every $a\in\Sigma$

\item $h(\{a\}\times I^{d_{cs}})$ is contained in the stable set of $h(a, 0)$, for every $a\in \Sigma$

\item $h(\Sigma\times \{0\})$ is a positive measure subset of some disk $D$ transverse to $\cW^{cs}$;


\item there is $K>0$ such that for any $u$-disks $D_1,D_2\subset
N$ which cross $h(\Sigma \times I^u)$, that is, $D_i$ intersects
$h(a\times I^{cs})$ for every $a\in \Sigma$, there is a holonomy
map $h^{cs}$ induced by $\cW^{cs}$ from $D_1\cap
h(a\times I^{cs})$ to $D_2\cap h(a\times I^{cs})$, the
Jacobian of the holonomy map between $\Leb_{D_1}$ and $\Leb_{D_2}$
is bounded by $K$ from above and $1/K$ from below.
\end{enumerate}
We also say that $\cB$ is a $cs$-block over the disk $D$ in (3).
If $D$ is contained in the unstable manifold of an index $d_{cs}$
periodic point $p$, then we say the $cs$-block is associated with p.

\begin{remark}\label{r.csboxpositivemeasure}
If $D$ is in the support of some ergodic Gibbs $u$-state
$m$ then $m(\cB)>0$: this is a consequence of the absolute
continuity property (4) and the fact that Gibbs $u$-states have
positive densities along strong unstable leaves
(Proposition~\ref{p.Gibbsustates}).
\end{remark}

We say that the $cs$-block has size $r>0$ if the \emph{plaque}
$h(\{a\}\times I^{d_{cs}})$ contains $\cW^{cs}_r(h(a, 0))$
for every $a\in \Sigma$.
If a map $\tilde{h}: \Sigma\times I^{d_{cs}}\to N$ satisfies
$$
\tilde{h}\mid {\Sigma\times \{0\}}\equiv h \mid {\Sigma\times\{0\}}
\quand
\tilde{h}(a\times I^{d_{cs}})\subset h(a\times I^{d_{cs}})
$$
for every $a\in\Sigma$ then
$\tilde\cB = \tilde{h}(\Sigma\times I^{d_{cs}})$ is called a sub-block of
$\cB$.

\begin{lemma}\label{l.uniform-csbox}
Let $m$ be an ergodic $u$-state of $f$ and $p\in \supp m$ be a periodic
point of stable index $d_{cs}$ whose stable manifold $\hW^s(p)$ has
size $r$. Then there is a $cs$-block associated with $p$ with size $r$.
\end{lemma}

\begin{proof}
By Lemma~\ref{l.csbox}, there is a $cs$-block over any $u$-disk
$D\subset\cW^u(p)$. Let $\kappa$ be the period of $p$.
For every large $n$, the backward image $f^{-n\kappa}(\cB)$
is a $cs$-block of size $r$ over the $u$-disk $f^{-n}(D)$.
\end{proof}

\subsection{Recurrence to $cs$-blocks}\label{ss.recurrence}

The next proposition is a key ingredient in the proof of our criteria for absolute continuity.

\begin{proposition}\label{p.csboxsection}
Let $m_i$, $i=1, \dots, s$ be the ergodic Gibbs $u$-states of $f$ and
$B_i$, $i=1, \dots, n$ be $cs$-blocks over $\cW^u$-disks $D_i\subset\supp m_i$.
Then for any positive Lebesgue measure subset $D^*$ of any $\cW^u$-disk $D$,
there exists $n>0$ arbitrarily large and there exists $1\leq i\leq s$ such that
$\Leb_{D}(D^* \cap f^{-n}(B_i))>0$.
\end{proposition}

\begin{proof}
(For notational simplicity, we use $m^u$ to denote $\Leb_{f^n(\Gamma)}$ for any
$u$-disk $\Gamma$ and any $n>0$.)
Let $D^*_\vep=\cup_{x\in D^*}B_\vep(x,D)$ where
$B_\vep(x,D)$ is the ball in $D$ with radius $\vep$ and center in
$x$, and $m$, $m_\vep$ be Cesaro accumulation points of the
iterates of Lebesgue measure on $D^*$ and $D^*_\vep$ respectively,
such that there is $\{n_j\}_{j=1}^\infty$ satisfying
$$
\lim_{j\to\infty}\frac{1}{n_jm^u(D^*)}\sum_{i=0}^{n_j-1}(f^i)_*m^u\mid {D^*}=m;
$$
$$
\lim_{j\to\infty}\frac{1}{n_jm^u(D^*_\vep)}\sum_{i=0}^{n_j-1}(f^i)_*m^u\mid {D^*_\vep}= m_\vep.
$$
Then they are $u$-states, denote by $m=a_1 m_1+\cdots+a_m m_s$
the ergodic decomposition of $m$, suppose $a_1\neq 0$.
For $\vep$ sufficiently small, $m^u(D^*_\vep)\approx m^u(D^*)$,
and then $m_\vep \approx m$, denote by
$m_\vep=a_{1,\vep} m_1+\cdots+a_{s,\vep} m_m$ the ergodic
decomposition of $m_\vep$, one has $a_{1,\vep}\approx a_1$.
Denote $D_1^*=D_1\cap B_1$ and
$D_{1,\delta}^*=\cup_{x\in D_1^*}B^u_\delta(x)$ and
$$
B_{1,\delta}^*=\{z: z\in \cW^{cs}_{loc}(x)\cap
\cW^u_{loc}(y) \text{ for } x\in D_{1,\delta}^* \text{ and } y\in B_1\}.
$$
Then there is $K_1>0$ such that for any $u$-disk $\Gamma$ which crosses
$B_{1,\delta}^*$, one has
$$\frac{m^u(\Gamma\cap B_1)}{m^u(\Gamma\cap
B_{1,\delta}^*)}>K_1,$$ that is because $m^u(\Gamma\cap
B_1)>\frac{1}{K}m^u(D_1\cap B_1)>0$ and $m^u(\Gamma\cap
B_{1,\delta}^*)$ is bounded above, where $K$ is the bound for the
Jacobian of the center stable foliation in $B_1$.
We can choose $\vep$ properly such that $m^u(\partial(D^*_{\vep}))=0$,
by Remark~\ref{r.csboxpositivemeasure}, suppose $m_1(B_1)=b_0>0$.
Because $B_{1,\delta}^*$ is open,
$$
\lim_{j\to\infty}\frac{1}{n_jm^u(D^*_\vep)}\sum_{i=0}^{n_j-1}
   \big(f^i_*m^u \mid D\big)\big(f^i(D^*_\vep)\cap B_{1,\delta}^*\big)\ge m_\vep(B_{1,\delta})
   \gtrsim b_0 a_\vep > \frac{a_1 b_0}{2}.
$$
So there is $n_j$ arbitrarily big such that
$$
\big(f^{n_j}_*m^u \mid D\big)\big(f^{n_j}(D^*_\vep)\cap B_{1,\delta}^*\big)
\ge \frac{a_1 b_0}{4} m^u(D_\vep^*).
$$
We claim that there is $b_1>0$ such that, for every $\vep>0$ sufficiently small,
$$
\big(f^{n_j}_*m^u \mid D\big) \big(f^{n_j}({D^*_\vep})\cap B_1\big)
\ge 2b_1 m^u(D_\vep^*).
$$
Let us prove the claim. For $\vep_1<\vep$, denote
$D^*_{\vep,\vep_1}=\{x\in D^*_\vep;d_D(x,\partial(D^*_\vep))>\vep_1\}$.
Then, for $\vep_1$ sufficiently small, one has $m^u({D^*_{\vep,\vep_1}})>m^u(D^*_\vep)-a_1b_0/8$.
It follows that
$$
\big(f^{n_j}_*m^u \mid D\big) \big(B_{1,\delta}^*\big)
\ge \frac{a_1 b_0}{8} m^u(D^*_\vep)
$$
and for any $x\in f^{n_j}(D^*_{\vep,\vep_1})\cap B_{1,\delta}^*$,
there is a $u$-disk $D_x\subset f^{n_j}(D^*_\vep)$ containing $x$
such that $D_x\cap \cW^{cs}_{loc}(y)\neq \emptyset$ for any $y\in
D_{1,\delta}^*$. Since
$$
\frac{m^u(D_x\cap B_1)}{m^u(D_x\cap B_{1,\delta}^*)}>K_1,
$$
and the distortion property on $\cW^u$, there is $K_2>0$ such that
$$
\frac{m^u(f^{-n_j}(D_x\cap B_1))}{m^u(f^{-n_j}(D_x\cap B_{1,\delta}^*))}>K_2.
$$
Then, taking $b_1={K_2a_1b_0}/{16}$, we get our claim:
$$
m^u(D^*_\vep\cap f^{-n_j}(B_1))>K_2m^u(D^*_{\vep,\vep_1}\cap f^{-n_j}(B_{1,\delta}^*))
> 2 b_1 m^u(D^*_\vep),
$$
Since $\lim_{\vep\to 0}m^u(D^*_\vep\setminus D^*)=0$, this proves that
$m^u(D^*\cap f^{-n_j}(B_1))>b_1 m^u(D^*)$. This completes the proof of the proposition.
\end{proof}

\begin{remark}\label{r.lowerbound}
Assuming there exists a unique Gibbs $u$-state, the arguments in the proof of Proposition~\ref{p.csboxsection}
yield a slightly stronger conclusion that will be useful in the sequel: there exists $b_1>0$ such that
for any positive Lebesgue measure subset $D^*$ of any $\cW^u$-disk $D$ and for any $1 \le i \le s$ there exist
arbitrarily large values of $n>0$ such that $\Leb_{D}(D^* \cap f^{-n}(B_i))\ge b_1 \Leb_D(D^*)$.
\end{remark}

\subsection{Upper leafwise absolute continuity}\label{ss.upperleafwise}

Here we prove Proposition~\ref{p.mainF}. Suppose there exists some measurable
set $Y$ with $\Leb(Y)>0$ that meets almost every center stable leaf $\cW^{cs}(z)$
on a zero $\Leb^{cs}$-measure subset. Up to replacing $Y$ by some full measure
subset, we may suppose that every $x\in f^n(Y)$ is a Lebesgue density point
of $f^n(Y)$ for every $n\geq 0$:
\begin{equation}\label{eq.dual}
\lim_{\rho\to 0}\frac{\Leb(B_\rho(x) \cap f^n(Y))}{\Leb(B_\rho(x))}=1.
\end{equation}
Since $f$ has finitely many ergodic $u$-states and their basins cover a full
measure subset of $N$ (see \cite{BoV00}), it is no restriction to suppose that $Y$
is contained in the basin of some ergodic Gibbs $u$-state $m$.
Let $\cB$ be a $cs$-block over some $u$-disk contained in the support of $m$
(recall Proposition~\ref{p.Gibbsustates} and Section~\ref{ss.localabscont}).
Since the strong unstable foliation is absolutely continuous (see \cite{BP74}), we
can find a $u$-disk $D$ such that $D^* = D \cap Y$ has positive $\Leb_D$-measure.
By Proposition~\ref{p.csboxsection}, there exists $n > 0$ such that
$\Leb_{f^n(D)}(f^n(D^*) \cap \cB)>0$. Take $y \in D^*$ such
that $f^n(y)\in \cB$ and $f^n(y)$ is a Lebesgue density point for
$f^n(D^*) \cap\cB$ inside $f^n(D)$. Then, for every small $\rho>0$,
$$
\frac{\Leb_{f^n(D)} \big(B^u_\rho(f^n(y)) \cap \cB\big)}
     {\Leb_{f^n(D)} \big(B^u_\rho(f^n(y)\big)}
\ge \frac{\Leb_{f^n(D)} \big(B^u_\rho(f^n(y)) \cap f^n(D^*) \cap \cB\big)}
         {\Leb_{f^n(D)} \big(B^u_\rho(f^n(y)\big)} \approx 1,
$$
where $B^u_\rho(x)$ denotes the connected component of $B_\rho(x) \cap \cW^u(x)$
that contains $x$. Then, since the center stable foliation is uniformly
absolutely continuous on the $cs$-block, there exists $c>0$ such that
$$
\frac{\Leb\big(B_\rho(f^n(y)) \cap \cB\big)}
     {\Leb\big(B_\rho(f^n(y)\big)}
     \ge c \quad\text{for all small } \rho>0.
$$
Together with \eqref{eq.dual}, this implies that $\Leb(f^n(Y) \cap \cB) > 0$.
On the other hand, the hypothesis implies that $f^n(Y)$ intersects almost
every center stable leaf on a zero Lebesgue measure subset. Using, once more,
that the center stable leaf is absolutely continuous on the $cs$-block,
we get that $\Leb(f^n(Y) \cap \cB)=0$. This contradicts the previous conclusion,
and that contradiction completes the proof of Proposition~\ref{p.mainF}.

\subsection{Non-expansion along the center}\label{ss.nonexpansion}

Now we prove Proposition~\ref{p.nonexpansion}.
Let the ergodic $u$-states $\{ m_i\}_{i=1}^m$, periodic points
$\{p_i\}_{i=1}^m,$ and $\rho,\vep$ given in the definition of
non-expansion along the center. By Lemma~\ref{l.uniform-csbox},
we can choose $cs$-blocks $\{\cB_i\}_{i=1}^m$ associated with $m_i$
with size $\rho$.

In order to prove the center stable foliation is absolutely
continuous, we just need show that for any two $u$-disks $D_1,D_2$
which are $\vep$ near, the holonomy map induced by $\cW^{cs}$
between $D_1$ and $D_2$ maps Lebesgue positive measure subset to a
Lebesgue positive measure subset, where two $u$-disks $D_1,D_2$ are
$\vep$ near if for any $x\in D_1$, there is $y\in D_2$ belonging
to $\cW^{cs}_{\vep} (x)$.

Suppose $D_1^*\subset D_1$ is a positive measure subset, denote by
$D_2^*\subset D_2$ the image of $D_1^*$ under cs-holonomy map.
Since $f$ is non-expanding along the center, we can assume that
for any $x\in D_1^*$, one has $f^n(\cW^{cs}_\vep(x))
\subset\cW^{cs}_\rho(f^n(x))$ for $n>0$. Choose $\tilde{B}_i$
a sub-block of $B_i$ with arbitrarily small size, then by
Proposition~\ref{p.csboxsection}, there is $n$ and $j$ such that
$m^u(f^n(D_1^*)\cap \tilde{B}_j)>0$. Because $f^n(D_1^*)$
and $f^n(D_2^*)$ are $\rho$ near, and the cs-holonomy map in $B_i$
is absolutely continuous, one has that $m^u(f^n(D_2^*)\cap
B_j)>0$, this implies $m^u(D_2^*)>0$, so $\cW^{cs}$ is absolutely
continuous.

\subsection{Center leaves in general position}\label{ss.newcriterion}

We are going to prove Proposition~\ref{p.generalposition}. Let us start by giving an
overview of the argument. We need to compare a set on any $u$-disk with its
projection to another $u$-disk under $cs$-holonomy. The idea is to consider
appropriate iterates of both $u$-disks intersecting a given $cs$-block, and then
take advantage of the uniform structure on the $cs$-block. The problem is that,
because $cs$-blocks have gaps along the center direction, one can not immediately
ensure that iterates of both disks intersect the same $cs$-block. To this end,
we use the twisting property in the assumption of general position to find a pair of
$cs$-blocks whose union covers the whole center direction, in the sense that it
intersects any large iterate of any $u$-disk. Then, we show that some iterate of
any of the disks intersects both $cs$-blocks, which gives the required property.

Now we fill in the details in the proof. Let $f$ and $\ell$ be as in the statement
of  the proposition. For simplicity, consider the center leaf $\ell$ to be fixed
(in other words, $\kappa=1$) and we also take the attractor $a$ and repeller $r$
of $f\mid\ell$ to be fixed. Extension to the general case is straightforward.

\begin{lemma}
The diffeomorphism $f$ has a unique ergodic $u$-state and its support contains
the attractor $a$.
\end{lemma}

\begin{proof}
By Lemma~\ref{l.disjoint}, the supports of all ergodic Gibbs $u$-states are pairwise disjoint.
Thus, it suffices to show that the support of any ergodic $u$-state contains $a$.
By Proposition~\ref{p.Gibbsustates}, the support of $m$ consists of entire unstable leaves.
So, it suffices to prove that every strong unstable leaf intersects the stable manifold
$\hW^s(a)$ of the attractor. By hypothesis, every strong unstable leaf intersects $\cW^s(\ell)$.
Moreover, $\cW^s(\ell)$ is the union of $\hW^s(a)$ with the strong stable leaf through the
repeller $r$. If a strong unstable leaf $L$ intersects $\cW^s(r)$ then its forward orbit
accumulates on $\cW^u(r)$ and, in particular, on $h^u(r)$. Since $\ell$ is in general position,
$h^u(r) \neq h^s(r)$ and so $h^u(r)$ belongs to the stable manifold of $a$.
Hence, in any case, $L$ does intersect $\hW^s(a)$. This completes the argument.
\end{proof}

Consider the four points $a_s=h^s(a)$, $a_u=h^u(a)$, $r_s=h^s(r)$, $r_u=h^u(r)$ in $\ell'$.
For $\rho$, $\vep>0$ small, and $\zeta\in\ell'$, denote
$$
\cW_{\rho}^s(\ell')=\cup_{\xi\in\ell'}\cW^s_{\rho}(\xi)
\quand
V^{cs}_{\vep}(\zeta)=\cup_{\xi\in B^c_\vep(\zeta)}\cW^s_{\rho}(\xi).
$$
Let $\tilde\cB$ be a $cs$-block over $\cW^u_{loc}(a)$ (Lemma~\ref{l.uniform-csbox}).
Then for $n$ large, $f^{-n}(\tilde\cB)$ intersects $\cW^u_{loc}(a_s)$ in a set $\tilde{D}_1^*$
with positive Lebesgue measure, and we may choose a $cs$-block $\cB_1\subset f^{-n}(\tilde\cB)$
over a $u$-disk $\tilde{D}_1 \supset \tilde{D}_1^*$
such that
$$
\cW^u_{2\tau}(\zeta) \cap \cB_1 \neq \emptyset
\quad\text{for all }
\zeta\in\cW^s_{\rho}(\ell')\setminus V^{cs}_\vep(r_s).
$$
We think of the union $W^u_{2\tau}(V^{cs}_\vep(r_s))$ of the local unstable manifolds through
the local center stable manifold of $r_s$ as the gap of $\cB_1$ along the center direction.
See Figure~\ref{f.compact}.

\begin{figure}[h]
\begin{center}
\psfrag{a}{$a$}\psfrag{r}{$r$}\psfrag{as}{$a_s$}\psfrag{rs}{$r_s$}\psfrag{au}{$a_u$}\psfrag{ru}{$r_u$}
\psfrag{L}{$\ell$}\psfrag{L2}{$\ell'$}\psfrag{s}{{\small $W^s$}}\psfrag{u}{{\small $W^u$}}
\psfrag{Q}{{\small $N/\cW^c$}}
\includegraphics[height=2in]{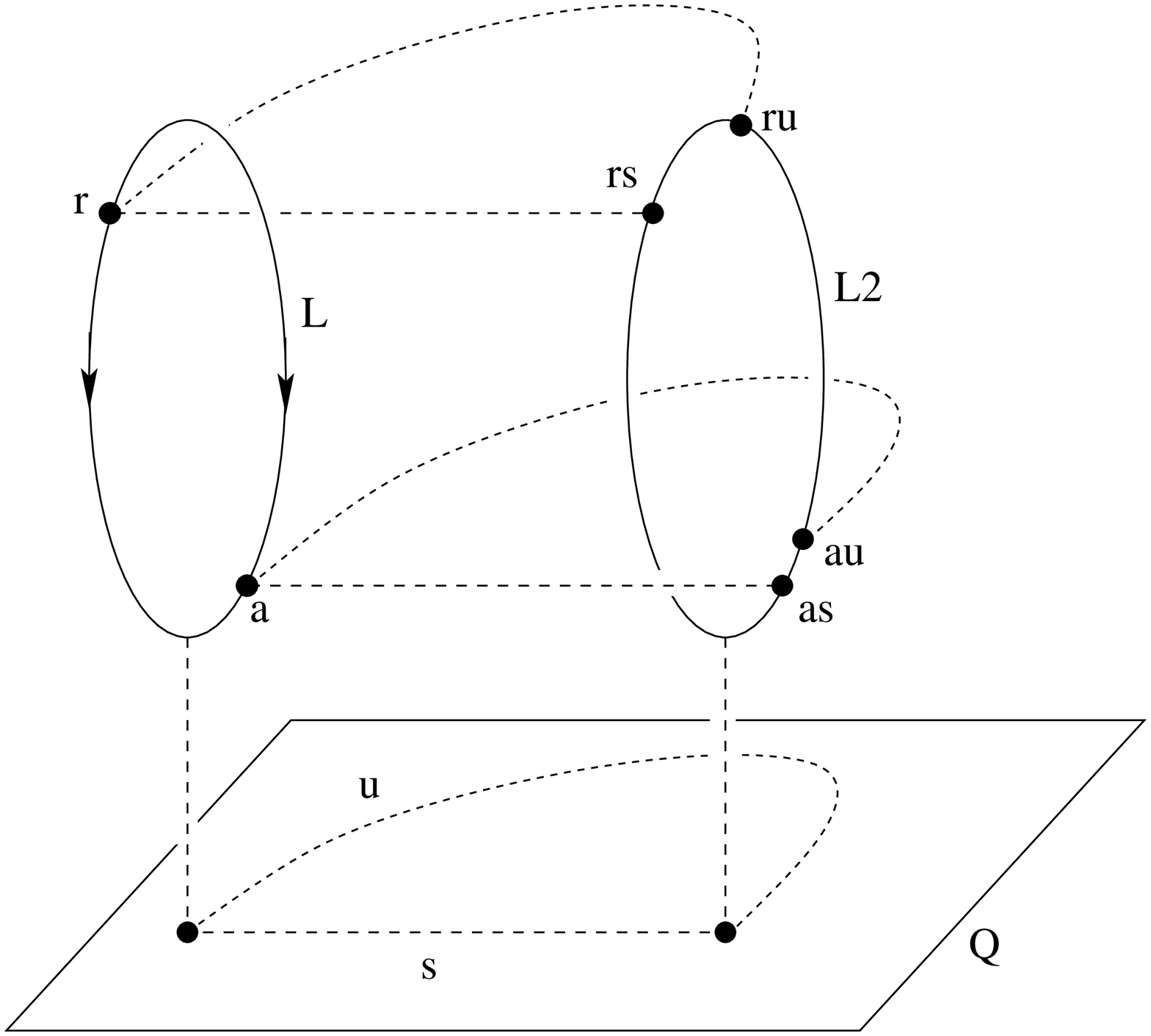}
\caption{\label{f.compact}}
\end{center}
\end{figure}

Dually, consider a $cs$-block $\cB_2\subset f^{-n}(\tilde\cB)$ over a $u$-disk
$\tilde{D}_2\subset\cW^u(a)$ such that
$$
\cW^u_{2\tau}(\zeta)\cap \cB_2\neq \emptyset
\quad\text{for all }
\zeta\in \cW^s_{\rho}(\ell')\setminus V^{cs}_\vep(r_u).
$$
Again, the union $W^u_{2\tau}(V^{cs}_\vep(r_u))$ of the local unstable manifolds through
the local center stable manifold of $r_u$ is the gap of $\cB_2$ along the center direction.
Moreover, we may fix $\delta_0>0$ such that, for any $\zeta\in \cW^s_{\rho}(\ell')$,
either
$$
m^u(\cW^u_{2\tau}(\zeta)\cap \cB_1)>\delta_0 \quad\text{or}\quad m^u(\cW^u_{2\tau}(\zeta)\cap \cB_2)>\delta_0.
$$
This is, in precise terms, what we meant when we announced that the union $\cB_1 \cup \cB_2$ of
the two $cs$-blocks would cover the whole center direction.

Now consider a new $cs$-block $\cB$ defined as the union of
$$
(\cW^u_{2\tau}(\xi)\cap \cB_2)\cup(\cW^u_{2\tau}(\xi)\cap \cB_1)
$$
over all $\xi\in W_\rho^s(\ell') \setminus \big(V^{cs}_\vep(r_s) \cup V^{cs}_\vep(r_u)\big)$.
In other words, $\cB$ is obtained from $\cB_1 \cup \cB_2$ by removing the two gaps.
Thus, $\cB=\cB^1\cup\cB^2$ with $\cB^1\subset\cB_1$ and $\cB^2\subset\cB_2$.
We are going to show that arbitrarily large iterates of any $u$-disk intersect both connected
components of $\cB$ on positive measure subsets.

\begin{lemma}\label{l.lem2}
Given any $u$-disk $D$ and any positive $\Leb_D$-measure subset $D^*$ there exists
$\zeta\in D^*$ and $k$ arbitrarily large such that
$$
\Leb_{f^{k}(D)}\big(\cW^u_{2\tau}(f^n(\zeta))\cap f^{k}(D^*)\cap \cB^i\big)>0
\quad\text{for both } i=1, 2.
$$
\end{lemma}

\begin{proof}
It is no restriction to suppose every point of $D^*$ is a Lebesgue density point.
Fix $\vep>0$ small (the precise choice will be given later).
Take any point $x\in D^*$ and let $r>0$ small enough so that $\Leb_{D}(D_r^*) > (1-\vep) \Leb_{D}(D_r)$,
where $D_r$ is the disk of radius $r$ around $x$ and $D_r^*=D_r \cap D^*$.
By Proposition~\ref{p.csboxsection} and Remark~\ref{r.uniqueustate} there exists $b_1>0$,
independent of $x$ and $r$, such that
$$
\Leb_{D}(D_r^* \cap f^{-n_i}(\cB^1))\ge b_1 \Leb_D(D_r^*) \ge b_1 (1-\vep) \Leb_D(D_r)
$$
for a sequence $n_i\to\infty$. Let $\rho>0$ be slightly smaller than $r$,
so that
$$
\Leb_D(D_\rho) > (1-\vep) \Leb_D(D_r).
$$
Then, for any $n_i$ sufficiently large and any $y\in D_\rho$, we have
$f^{-n_i}(W^u_{loc}(f^{n_i}(y))) \subset D_r$.
Since the local unstable manifold of $f^{n_i}(y)$ cuts across both $\cB^1$ and $\cB^2$,
this means that we can associate to $y \in D^*_\rho \cap f^{-n_i}(\cB_1)$ the following
subsets of $D_r$:
$$
D^1_i(y)=f^{-n_i}(W^u_{loc}(f^{n_i}(y))\cap\cB^1) \quand
D^2_i(y)=f^{-n_i}(W^u_{loc}(f^{n_i}(y))\cap\cB^2).
$$
By bounded distortion,
there exists $\kappa=\kappa(f)>0$ such that
$$
\Leb_D(D^2_i(y)) \ge \kappa \Leb_D(D^1_i(y))
\quad\text{for every $y$ and every $i$.}
$$
We also denote by $D^1_i$ and $D^2_i$ the (disjoint) unions of $D^1_i(y)$ and $D^2_i(y)$,
respectively, over all $y \in D^*_\rho \cap f^{-n_i}(\cB_1)$.
Then, the previous inequality gives
$$
\Leb_D(D^2_i) \ge \kappa \Leb_D(D^1_i)
\quad\text{for every $i$.}
$$
By Proposition~\ref{p.csboxsection} and Remark~\ref{r.lowerbound}, there exists a sequence
$(n_i)_i$ of positive integers and there exists $b_1>0$ such that
$$
\Leb_{D}(D_\rho^* \cap f^{-n_i}(\cB^1))
\ge b_1 \Leb_{D_\rho}(D_\rho^*)
\ge b_1 (1-\vep)^2\Leb_{D}(D_\rho).
$$
Consequently,
$$
\Leb_{D}(D^1_i)
\ge b_1 (1-\vep)^2\Leb_{D}(D_\rho)
\ge b_1 (1-\vep)^3\Leb_{D}(D_r).
$$
This implies that $\Leb_D(D^2_i) \ge b_2 \Leb_{D}(D_r)$, where the constant $b_2>0$ is
independent of $i$ and the choice of $r$. Now, suppose the lemma is false.
Then $D^2_i(y) \cap D^*$ is empty, for every $y\in D_\rho^* \cap f^{-n_i}(\cB^1)$,
that is, $D^2_i\cap D^*=\emptyset$. It follows that $\Leb_D(D_r^*)\le (1-b_2) \Leb_D(D_r)$.
This contradicts the choice of $D_r^*$ at the beginning of the proof, as long as we fix
$\vep < b_2$. The proof of the lemma is complete.
\end{proof}

\begin{proof}[Proof of Proposition~\ref{p.generalposition}]
Let $h^{cs}:D^1\to D^2$ be a $cs$-holonomy between $u$-disks $D_1$ and $D_2$.
Let $D_1^*\subset D_1$ be a positive $\Leb_{D_1}$-measure subset and
$D_2^*=h^{cs}(D_1^*)$. We want to prove that $\Leb_{D_2}(D_2^*)$ is also positive.
By Lemma~\ref{l.lem2}, there exists $\zeta\in D^*_1$ and $k\ge 1$ such that
\begin{equation}\label{eq.exampleIII}
\Leb_{f^{k}(D)}\big(\cW^u_{2\tau}(f^{k}(\zeta))\cap f^{k}(D^*_1)\cap \cB^i\big) > 0
\quad\text{for both } i=1, 2.
\end{equation}
Notice that for $k$ big enough, $\cW^u_{2\tau}(f^{k}(\zeta))$ and
$\cW^u_{2\tau}(f^{k}(h^{cs}(\zeta)))$ are contained in nearby $cu$-disks.
That is because the stable foliation is uniformly contracting.
Then $\cW^u_{2\tau}(h^{cs}(\zeta))\cap \cW^s_{\rho}(\ell')\neq \emptyset$.
This implies $\cW^u_{2\tau}(h^{cs}(\zeta))\cap \tilde{\cB}_1\neq \emptyset$
or $\cW^u_{2\tau}(h^{cs}(\zeta))\cap \cB_2\neq \emptyset$.
Since $\tilde{\cB}_1, {\cB}_2$ are $cs$-blocks, whose $cs$-foliations are
uniformly absolutely continuous, from \eqref{eq.exampleIII} one gets that
$$
\Leb_{f^{k}(D_2)}\big(\cW^u_{2\tau}(f^{k}(h^{cs}(\zeta)))\cap f^{k}(D_2^*)\cap {\cB^i}\big)>0
$$
for either $i=1$ or $i=2$. This implies that $\Leb_{D_2}(D_2^*)>0$.
Thus, the center stable foliation is absolutely continuous, as claimed.
\end{proof}

\section{Robust absolute continuity}\label{s.robustabscont}

Here we use the results in the previous section to give examples of open sets
of diffeomorphisms with absolutely continuous center stable/unstable foliations.

\subsection{Ma\~n\'e's example}\label{ss.Mane}

Ma\~n\'e~\cite{Man78} constructed a $C^1$ open set of diffeomorphisms $\cU$ such
that every $f\in \cU$ is partially hyperbolic (but not hyperbolic), dynamically coherent,
and transitive. From Proposition~\ref{p.nonexpansion} one gets that every $C^k$, $k>1$
diffeomorphism $f$ in some non-empty $C^1$ open subset $\cU'$ has absolutely continuous
center stable foliation. To explain this, let us recall some main features in Ma\~n\'e's
construction.

One starts from a convenient linear Anosov map $A:\TT^3\to \TT^3$ with eigenvalues
$0<\lambda_1<\lambda_2<1<\lambda_3$. Let $p$ be a fixed point of $A$ and $\rho>0$ be small.
One deforms $A$ inside the $\rho$-neighborhood of $p$, so as to create some fixed point with
stable index $1$, while keeping the diffeomorphism unchanged outside $B_\rho(p)$.
Ma\~n\'e~\cite{Man78} shows that this can be done in such a way that the diffeomorphism
$f_0:\TT^3\to\TT^3$ thus obtained is partially hyperbolic, with splitting $E^s\oplus E^c \oplus E^s$
where all factors have dimension $1$, and every diffeomorphism in some $C^1$ neighborhood $\cU$
is dynamically coherent and transitive. The presence of periodic points with both stable indices
$1$ and $2$ ensures that $f_0$ is not Anosov. Bonatti, Viana~\cite{BoV00} observed that every $C^k$,
$k>1$ diffeomorphism $f\in \cU$ has mostly contracting center direction.
Here, as well as in the steps that follow, one may have to reduce the neighborhood $\cU$.
Then Bonatti, D\'\i az, Ures \cite{BDU02} showed that the unstable foliation of every $f\in \cU$
is minimal. According to~\cite{BoV00}, this implies that every $C^k$, $k >1$ diffeomorphism
$f\in\cU$ admits a unique physical measure, whose basin contains Lebesgue almost every point.
The non-expansion condition in Proposition~\ref{p.nonexpansion} can be checked as follows.

A crucial observation is that the center stable bundle $E^c\oplus E^s$ is uniformly contracting
outside $B_\rho(p)$, for all diffeomorphisms in a neighborhood, because $f_0=A$ outside $B_\rho(p)$.
Let $q$ be another fixed or periodic point of $A$ and assume $\rho$ was chosen much smaller
than the distance from $p$ to the orbit of $q$. Then $q$ remains a periodic point for $f_0$,
with stable index $2$ and stable manifold of size $\ge 5 \rho$. Let $q_f$ denote the hyperbolic
continuation of $q$ for every $f$ in a neighborhood of $f_0$: $q_f$ is a periodic point
with stable index $2$ and stable manifold of size $\ge 4 \rho$. The fact that $E^c\oplus E^s$
is uniformly contracting outside $B_{\rho}(p)$ also implies that
$f^n(\cW^{cs}_{\rho}(x))\subset \cW^{cs}_{2\rho}(f^n(x))$ for all $x\in \TT^3$ and $n\ge 0$.
This proves that $f$ is non-expanding along the center direction, and so we may apply
Proposition~\ref{p.nonexpansion} to conclude that the center stable foliation of every
$f$ near $f_0$ is absolutely continuous.

We ignore whether the center unstable foliation and the center foliation are absolutely
continuous or not in this case. However, in the next section, a different construction
allows us to give examples where all three invariant foliations are robustly absolutely
continuous.

\subsection{Robust absolute continuity for all invariant foliations}\label{ss.proofofmainG}

Here we prove Theorem~\ref{t.mainG} and use it to deduce Theorem~\ref{t.main2}. We begin
with an intermediate result:

\begin{proposition}\label{p.intermediate}
Let $f_0:N\to N$ be a $C^k$, $k>1$ skew-product $f_0(x,\theta)=(g_0(x),h_0(x,\theta))$,
where $g_0$ is a transitive Anosov diffeomorphism. Assume that $f_0$ is accessible and
has some periodic center leaf in general position. Then there exists a $C^k$ neighborhood
$\cV$ of $f_0$ such that for every $f\in\cV$, the center stable, center unstable, and
center foliation are absolutely continuous.
\end{proposition}

\begin{proof}
Every skew-product has absolutely continuous center stable and center unstable foliation
and is robustly dynamically coherent (by \cite{HPS77}; the center foliation of a partially
hyperbolic skew-product is always plaque expansive). In particular,
$f_0$ satisfies all the hypotheses of Theorem~\ref{t.mainB}. The presence of a
Morse-Smale center leaf prevents $f_0$ from being conjugate to a  rotation extension.
Thus, the center direction is mostly contracting in a whole neighborhood of $f_0$.
The assumption that $g_0$ is transitive also ensures that every strong unstable leaf
intersects $\cW^s(\ell)$. So, we are in a position to apply Corollary~\ref{c.generalposition}
to conclude that the center stable foliation is robustly absolutely continuous.
The same reasoning applied to the inverse of $f_0$ gives that the center unstable
foliation is also robustly absolutely continuous. From the following general fact
we get that the center foliation is also robustly absolutely continuous:

\begin{lemma}[Pugh, Viana, Wilkinson~\cite{PVW}]\label{l.csufoliations}
Let $\cF^1$, $\cF^2$, $\cF^3$ be foliation in some smooth manifold $N$ such that
$\cF^1$ and $\cF^2$ are transverse at every point and the leaves of $\cF^3$ are
coincide with the intersections of leaves of $\cF^1$ and $\cF^2$:
for every point $x\in N$, $\cF^3(x)=\cF^1(x)\cap \cF^2(x)$. If $\cF_1$ and $\cF_2$
are absolutely continuous then so is $\cF_3$.
\end{lemma}

\begin{proof}
Suppose $D_1,D_2$ are two disks transverse with $\cF^3$, and $h^3:
D_1\rightarrow D_2$ is the holonomy map induced by $\cF^3$. Then
$\cF^1$ and $\cF^2$ induce two foliations $\hat{\cF}^1_i$ and
$\hat{\cF}^2_i$ on $D_i$, $i=1,2$, and these two foliations
absolutely continuous in $D_i$. Fix $l_1\subset D_1$ a leaf of
$\hat{\cF}^2_1$, and denote by $l_2=h^3(l_1)$, then $l_2$ is a
leaf of $\hat{\cF}^2_2$.
Since the foliations $\hat{\cF}^1_{i}$, $i=1,2$ are absolutely
continuous, one has that the disintegration of the Lebesgue
measure $\Leb_{D_1}$ along the foliation $\hat{\cF}^1_1$ is
$$
\Leb_{D_1}=\varphi_x(y)d\Leb_{\hat{\cF}^1_1(x)}(y)d\Leb_{l_1}(x),
\quad\text{where } \varphi_x(y)>0
$$
and the disintegration of the Lebesgue measure of $D_2$ along the
foliation $\hat{\cF}^1_2$ is
$$
\Leb_{D_2}=\phi_x(y)d\Leb_{\hat{\cF}^1_2(x)}(y)d\Leb_{l_2}(x),
\quad\text{where } \phi_x(y)>0.
$$
Now for any set $\Delta_1\subset D_1$ with $\Leb_{D_1}(\Delta_1)>0$,
denote its image for $h^3$ by $\Delta_2$. By the above formulas
for the disintegration, there is a positive $\Leb_{l_1}$ measure subset
$\Gamma_1\subset l_1$ such that for any $x\in \Gamma_1$, one has
$$
\Leb_{\hat{F}^1_1(x)}(\Delta_1\cap \hat{F}^1_1(x))>0.
$$
Denote $\Gamma_2=h^3(\Gamma_1)\subset l_2$.
By the absolute continuity of $\cF^1$ and $\cF^2$, $\Leb_{l_2}(\Delta_2)>0$
and $\Leb_{\hat{F}^1_2(x)}(\Delta_2\cap \hat{F}^1_2(x))>0$ 
for any $x\in \Gamma_2$. This implies
$\Leb_{D_2}(\Delta_2)>0$, and so the proof is complete.
\end{proof}
This completes the proof of Proposition~\ref{p.intermediate}.
\end{proof}


To complete the proof of Theorem~\ref{t.mainG} it suffices to note that
any skew-product $f_0$ with a Morse-Smale center leaf, as in the
statement of the theorem, is approximated by skew-products with center
leaves in general position: all that is missing is property (b) in the
definition of general position, and this can be achieved by a $C^k$
small perturbation inside the space of skew-products.
Then Theorem~\ref{t.mainG} follows from Proposition~\ref{p.intermediate}.

Now Theorem~\ref{t.main2} is deduced as follows. For any skew-product $f_0$
as in the statement, Let $\cU$ be an open set that accumulates on $f_0$ as
given by Theorem~\ref{t.mainG}: for any $f\in\cU$ the center stable, 
center unstable, and center foliations are absolutely continuous.
Then let $\cV\subset\cU$ be an open subset such that every $f\in\cV$ is
accessible (\cite{NT01}). Then every $f\in\cV$ has finitely many physical 
measures, with basins containing almost every point. The Morse-Smale behavior
on the center leaf $\ell$ prevents $f$ from being conjugate to a rotation extension.
Thus, we are in case (b) of Theorem~\ref{t.mainB}. From the fact that $\ell$
contains a unique periodic attractor we also get that the physical measure is
unique (see Theorem~\ref{t.mainD}). The same argument applies for $f^{-1}$.
This finishes the proof of Theorem~\ref{t.main2}.

\subsection{Volume preserving systems}\label{ss.conservative}

Here we prove Theorem~\ref{t.main3} and a pair of related results.
Based on these, we also describe a, partially conjectural, scenario for absolute 
continuity of foliations of conservative and dissipative systems.

Part (1) of Theorem~\ref{t.main3} is a direct consequence of the main
result of Baraviera, Bonatti~\cite{BB03}. Part (2) is given by the following result:

\begin{lemma}\label{l.ex4}
For any $f\in\cW_0$ with $\lambda^c(f)>0$, the center foliation and the 
center stable foliation are not upper leafwise absolutely continuous.
\end{lemma}

\begin{proof}
Fix $c\in (0,\lambda^c(f))$. Then, by the Birkhoff ergodic theorem, the set
$$
\Gamma_{c,1}
 =\{x\in N:\lim \frac{1}{n}\sum_{i=1}^n \log\|Df^{-1}\mid {E^c(f^{i}(x))}\|^{-1}\ge c\}
$$
has positive volume. Then, by Proposition~\ref{p.finiteGamma}, there is $n_0\ge 1$ such
that the intersection of any center leaf with $\Gamma_{c,1}$ has at most $n_0$ points.
In particular, the intersection has zero volume inside the center leaf. 
So, the center foliation of $f$ is not upper leafwise absolutely continuous.
Next, observe that the set $\Gamma_{c,1}$ consists of 
entire strong stable leaves. So, the intersection of $\Gamma_{c,1}$ with any center
stable leaf consists of no more than $n_0$ strong stable leaves. This implies
that the intersection has zero volume inside the center stable leaf. Consequently,
the center stable foliation is not upper leafwise absolutely continuous.
In particular, we get that the center foliation and the center stable foliation are
not absolutely continuous, as claimed.
\end{proof}

Now we prove part (3) of the theorem. Let $p\in M$ be a periodic point of $g_0$ and
$a\in M$ be a homoclinic point associated to $p$. For simplicity, we take the periodic
point to be fixed. Let us begin by constructing $\cW_1$. The first step is to
approximate $f_0$ by some diffeomorphism $f_1$ such that $\lambda^c(g)>0$ for any 
$g$ in a $C^1$ neighborhood. This can be done by the perturbation method in \cite{BB03};
the perturbation may be chosen such that $f_1=f_0$ on a neighborhood of $\{p\}\times S^1$,
and we assume that this is the case in what follows. The second step is to find $f_2$
arbitrarily close to $f_1$ such that, denoting by $\ell_p$ and $\ell_a$ the center
leaves associated to the continuation of $p$ and $a$,
\begin{itemize}
\item every strong unstable leaf of $f_2$ intersects $W^s(\ell_p)$;
\item the restriction of $f_2$ to $\ell_p$ is a Morse-Smale diffeomorphism,
with a single attractor $\xi$ and a single repeller $\eta$
\item and $\cW^u(\eta)$ and $W^s(\xi)$ are in general position (we call this non-strong connection).
\end{itemize}
These properties remain valid in a small neighborhood of $f_2$. As a final step, 
we use \cite{BDU02,HHU07p,HHTU10} to find a diffeomorphism $f_3$ arbitrarily close
to $f_2$ and such that the strong stable and the strong unstable foliations are 
minimal in a whole $C^1$ neighborhood of $f_3$. We take $\cW_1$ to be such a 
neighborhood. By \cite{BoV00}, for every diffeomorphism $f\in\cW$ the inverse
$f^{-1}$ has mostly center direction. Then, by \cite{An10}, the same is true
in a whole $C^k$ neighborhood $\cW_f$ in the space of all (possibly dissipative).
diffeomorphisms. Hence, we are in a position to apply Corollary~\ref{c.generalposition}
to conclude that the center unstable foliation is absolutely continuous for 
every diffeomorphism in $\cW_f$.

This finishes the proof of Theorem~\ref{t.main3}. The next proposition is a variation of 
results in \cite{AVW1} where center foliations are replaced by center stable or center 
unstable foliations.

\begin{proposition}\label{p.variation1}
Let $f_0$ be as in Theorem~\ref{t.main1}, where $M$ is a surface, and let $f$ be any 
$C^1$ nearby accessible, volume preserving diffeomorphism with $\lambda^c(f)=0$.
If either the center stable foliation or the center unstable foliation is absolutely
then $f$ is smoothly conjugate to a rotation extension and the center foliation is a 
smooth foliation.
\end{proposition}

\begin{proof}
Suppose $\cW^{cs}$ is absolutely continuous. Then we may apply Theorem~\ref{t.mainB}.
In this case Lebesgue measure is a Gibbs $u$-state with zero center exponent,
and so we are in the elliptic case (a) of the theorem. In particular, the center
foliation is leafwise absolutely continuous. Then we can apply \cite{AVW1} to conclude
that the center foliation is smooth and $f$ is smoothly conjugate to a rigid model.
In present case, where the center fiber bundle is trivial, we get that $f$ is 
topologically conjugate to a rotation extension (cf. Remark~\ref{r.uniqueustate}).
\end{proof}

\begin{remark}\label{r.variation2}
Suppose $f$ is partially hyperbolic, dynamically coherent, volume preserving,
and all the center exponents are negative at almost every point. 
Then the center stable foliation of $f$ is upper leafwise absolutely continuous.
This is a fairly direct consequence of Pesin theory. Indeed, if all the 
Lyapunov exponents are negative then the Pesin local stable manifold of almost 
every point is a neighborhood of the point inside its center stable leaf.
Then the absolute continuity of Pesin laminations~\cite{Pe76} implies that 
the center stable foliation is upper leafwise absolutely continuous.
\end{remark}

We close with a couple of conjectures on the issue of absolute continuity.
The first one deals with dissipative systems.

\begin{conjecture}\label{cj.1}
Let $k>1$ and $\cC_k$ be the space of partially hyperbolic, dynamically coherent $C^k$
diffeomorphisms with mostly contracting center direction. Then, for an open and
dense subset,
\begin{itemize}
\item if there is a unique physical measure then
      the center stable foliation is absolutely continuous;
\item if there is more than one physical measure
      then the center stable foliation is not upper leafwise absolutely continuous.
\end{itemize}
\end{conjecture}
Examples of the second situation will appear in a forthcoming paper~\cite{VY2}.

\begin{figure}[h]
\begin{center}
\psfrag{+}{}
\psfrag{+1}{\footnotesize $\lambda^c>0$: $\cW^c$ and $\cW^{cs}$ not abs cont}
\psfrag{+2}{\qquad\footnotesize $\cW^{cu}$ abs cont, generically ?} 
\psfrag{-}{}
\psfrag{-1}{\footnotesize $\lambda^c<0$: $\cW^c$ and $\cW^{cu}$ not abs cont}
\psfrag{-2}{\qquad\footnotesize $\cW^{cs}$ abs cont, generically ?} 
\psfrag{0}{\footnotesize $\lambda^c=0$}
\psfrag{r1}{\footnotesize all foliations smooth}
\psfrag{r2}{\footnotesize (rigidity)}
\psfrag{a}{\footnotesize all foliations generically}
\psfrag{a1}{\footnotesize  not abs cont}
\psfrag{b}{}
\psfrag{b1}{}
\includegraphics[height=2in]{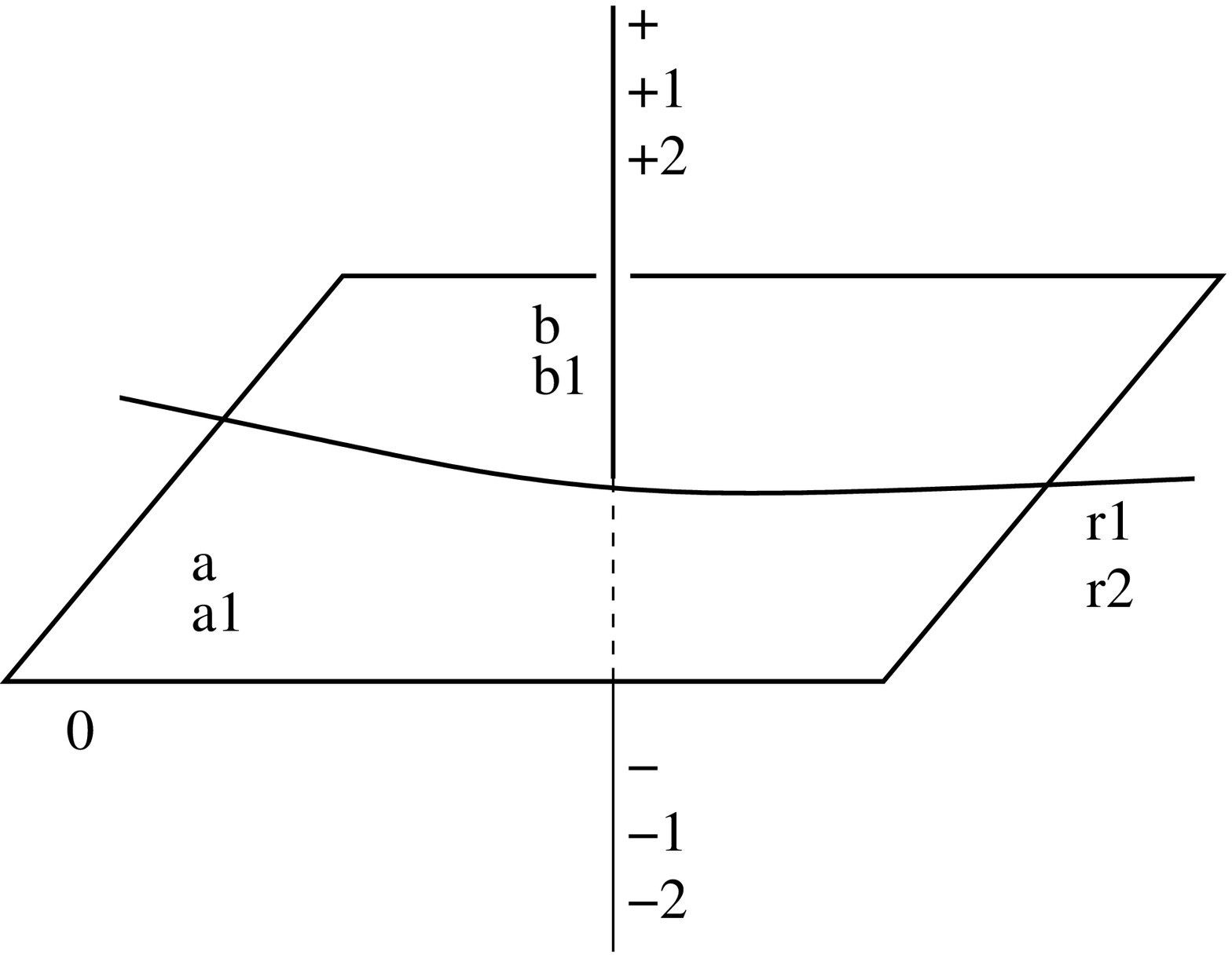}
\caption{\label{f.scenario}}
\end{center}
\end{figure}

\begin{conjecture}\label{cj.2}
Let $k>1$ and $\cV_k$ be the space of partially hyperbolic, dynamically coherent,
volume preserving $C^k$ diffeomorphisms whose center Lyapunov exponents are
negative at almost every point. Then, for an open and dense subset, the center 
stable foliation is absolutely continuous.
\end{conjecture}


Figure~\ref{f.scenario} outlines a scenario for these issues in a relevant special
case, namely near the map $f_0=g_0\times\id$ as in Theorem~\ref{t.main1}.
Accessibility is assumed throughout (but is not needed for the negative results in
$\lambda^c\neq 0$). Generically means for open and dense in $C^k$ topology, $k\ge 1$.
Upper leafwise absolute continuity of the center unstable is known for $\lambda^c>0$,
as we have seen, and we have also found an open subset with (full) absolute 
continuity of the center unstable.

\bibliographystyle{plain}
\bibliography{bib}

\begin{thebibliography}{10}

\bibitem{AbV}
F.~Abdenur and M.~Viana.
\newblock Flavors of partial hyperbolicity.
\newblock Preprint www.preprint.impa.br 2009.

\bibitem{Al00}
J.~F. Alves.
\newblock {SRB} measures for non-hyperbolic systems with multidimensional
  expansion.
\newblock {\em Ann. Sci. {\'E}cole Norm. Sup.}, 33:1--32, 2000.

\bibitem{Al03}
J.~F. Alves.
\newblock {\em Statistical analysis of non-uniformly expanding dynamical
  systems}.
\newblock Lecture Notes 24th Braz. Math. Colloq. IMPA, Rio de Janeiro, 2003.

\bibitem{AA04}
J.~F. Alves and V.~Ara{\'u}jo.
\newblock Hyperbolic times: frequency versus integrability.
\newblock {\em Ergodic Theory Dynam. Systems}, 24:329--346, 2004.

\bibitem{ABV00}
J.~F. Alves, C.~Bonatti, and M.~Viana.
\newblock S{RB} measures for partially hyperbolic systems whose central
  direction is mostly expanding.
\newblock {\em Invent. Math.}, 140:351--398, 2000.

\bibitem{ALP03}
J.~F. Alves, S.~Luzzatto, and V.~Pinheiro.
\newblock Markov structures for non-uniformly expanding maps on compact
  manifolds in arbitrary dimension.
\newblock {\em Electron. Res. Announc. Amer. Math. Soc.}, 9:26--31
  (electronic), 2003.

\bibitem{ALP05}
J.~F. Alves, S.~Luzzatto, and V.~Pinheiro.
\newblock Markov structures and decay of correlations for non-uniformly
  expanding dynamical systems.
\newblock {\em Ann. Inst. H. Poincar\'e Anal. Non Lin\'eaire}, 22:817--839,
  2005.

\bibitem{AlV02}
J.~F. Alves and M.~Viana.
\newblock Statistical stability for robust classes of maps with non-uniform
  expansion.
\newblock {\em Ergodic Theory Dynam. Systems}, 22:1--32, 2002.

\bibitem{An10}
M.~Andersson.
\newblock Robust ergodic properties in partially hyperbolic dynamics.
\newblock {\em Trans. Amer. Math. Soc.}, 362:1831--1867, 2010.

\bibitem{An67}
D.~V. Anosov.
\newblock Geodesic flows on closed {R}iemannian manifolds of negative
  curvature.
\newblock {\em Proc. Steklov Math. Inst.}, 90:1--235, 1967.

\bibitem{Av09-dens}
A.~Avila.
\newblock Density of positive {L}yapunov exponents for quasiperiodic {${\rm
  SL}(2,\mathbb{R})$}-cocycles in arbitrary dimension.
\newblock {\em J. Mod. Dyn.}, 3:631--636, 2009.

\bibitem{ASV}
A.~Avila, J.~Santamaria, and M.~Viana.
\newblock Cocycles over partially hyperbolic maps.
\newblock Preprint www.preprint.impa.br 2008.

\bibitem{AV3}
A.~Avila and M.~Viana.
\newblock Extremal {L}yapunov exponents: an invariance principle and
  applications.
\newblock {\em Inventiones Math.}, 181:115--178, 2010.

\bibitem{AVW1}
A.~Avila, M.~Viana, and A.~Wilkinson.
\newblock Absolute continuity, {L}yapunov exponents, and rigidity.
\newblock Preprint www.preprint.impa.br 2009.

\bibitem{BB03}
A.~Baraviera and C.~Bonatti.
\newblock Removing zero central {L}yapunov exponents.
\newblock {\em Ergod. Th. {\&} Dynam. Sys.}, 23:1655--1670, 2003.

\bibitem{BDU02}
C.~Bonatti, L.~J. D{\'\i}az, and R.~Ures.
\newblock Minimality of strong stable and unstable foliations for partially
  hyperbolic diffeomorphisms.
\newblock {\em J. Inst. Math. Jussieu}, 1:513--541, 2002.

\bibitem{Beyond}
C.~Bonatti, L.~J. D{\'{\i}}az, and M.~Viana.
\newblock {\em Dynamics beyond uniform hyperbolicity}, volume 102 of {\em
  Encyclopaedia of Mathematical Sciences}.
\newblock Springer-Verlag, 2005.

\bibitem{BGV03}
C.~Bonatti, X.~G{\'o}mez-Mont, and M.~Viana.
\newblock G\'en\'ericit\'e d'exposants de {L}yapunov non-nuls pour des produits
  d\'eterministes de matrices.
\newblock {\em Ann. Inst. H. Poincar\'e Anal. Non Lin\'eaire}, 20:579--624,
  2003.

\bibitem{BoV00}
C.~Bonatti and M.~Viana.
\newblock S{RB} measures for partially hyperbolic systems whose central
  direction is mostly contracting.
\newblock {\em Israel J. Math.}, 115:157--193, 2000.

\bibitem{BP74}
M.~Brin and Ya. Pesin.
\newblock Partially hyperbolic dynamical systems.
\newblock {\em Izv. Acad. Nauk. SSSR}, 1:177--212, 1974.

\bibitem{BS02}
M.~Brin and G.~Stuck.
\newblock {\em Introduction to dynamical systems}.
\newblock Cambridge University Press, 2002.

\bibitem{BW08a}
K.~Burns and A.~Wilkinson.
\newblock On the ergodicity of partially hyperbolic systems.
\newblock {\em Annals of Math.}

\bibitem{BST}
J.~Buzzi, O.~Sester, and M.~Tsujii.
\newblock Weakly expanding skew-products of quadratic maps.
\newblock {\em Ergod. Th. {\&} Dynam. Sys.}, 23:1401--1414, 2003.

\bibitem{DW03}
D.~Dolgopyat and A.~Wilkinson.
\newblock Stable accessibility is {$C^1$} dense.
\newblock {\em Ast\'erisque}, 287:33--60, 2003.

\bibitem{Gou06}
S.~Gouezel.
\newblock Decay of correlations for nonuniformly expanding systems.
\newblock {\em Bull. Soc. Mat. France}, 134:1--31, 2006.

\bibitem{HHTU10}
F.~Rodriguez Hertz, M.~A.~Rodriguez Hertz, A.~Tahzibi, and R.~Ures.
\newblock Creation of blenders in the conservative setting.
\newblock {\em Nonlinearity}, 23:211--223, 2010.

\bibitem{HHU07p}
F.~Rodriguez Hertz, M.~A.~Rodriguez Hertz, and R.~Ures.
\newblock Some results on the integrability of the center bundle for partially
  hyperbolic diffeomorphisms.
\newblock In {\em Partially hyperbolic dynamics, laminations, and
  {T}eichm\"uller flow}, volume~51 of {\em Fields Inst. Commun.}, pages
  103--109. Amer. Math. Soc., 2007.

\bibitem{HPS77}
M.~Hirsch, C.~Pugh, and M.~Shub.
\newblock {\em Invariant manifolds}, volume 583 of {\em Lect. Notes in Math.}
\newblock Springer Verlag, 1977.

\bibitem{Ka80}
A.~Katok.
\newblock Lyapunov exponents, entropy and periodic points of diffeomorphisms.
\newblock {\em Publ. Math. IHES}, 51:137--173, 1980.

\bibitem{Man78}
R.~Ma{\~{n}}{\'{e}}.
\newblock Contributions to the stability conjecture.
\newblock {\em Topology}, 17:383--396, 1978.

\bibitem{Man88}
R.~Ma{\~{n}}{\'{e}}.
\newblock A proof of the ${C}^1$ stability conjecture.
\newblock {\em Publ. Math. I.H.E.S.}, 66:161--210, 1988.

\bibitem{NT01}
V.~Ni{\c{t}}ic{\u{a}} and A.~T{\"o}r{\"o}k.
\newblock An open dense set of stably ergodic diffeomorphisms in a neighborhood
  of a non-ergodic one.
\newblock {\em Topology}, 40:259--278, 2001.

\bibitem{OW98}
D.~Ornstein and B.~Weiss.
\newblock On the {B}ernoulli nature of systems with some hyperbolic structure.
\newblock {\em Ergodic Theory Dynam. Systems}, 18:441--456, 1998.

\bibitem{PS82}
Ya. Pesin and Ya. Sinai.
\newblock {G}ibbs measures for partially hyperbolic attractors.
\newblock {\em Ergod. Th. {\&} Dynam. Sys.}, 2:417--438, 1982.

\bibitem{Pe76}
Ya.~B. Pesin.
\newblock Families of invariant manifolds corresponding to non-zero
  characteristic exponents.
\newblock {\em Math. USSR. Izv.}, 10:1261--1302, 1976.

\bibitem{Pi06}
V.~Pinheiro.
\newblock Sinai-{R}uelle-{B}owen measures for weakly expanding maps.
\newblock {\em Nonlinearity}, 19:1185--1200, 2006.

\bibitem{PS89}
C.~Pugh and M.~Shub.
\newblock Ergodic attractors.
\newblock {\em Trans. Amer. Math. Soc.}, 312:1--54, 1989.

\bibitem{PS97}
C.~Pugh and M.~Shub.
\newblock Stably ergodic dynamical systems and partial hyperbolicity.
\newblock {\em J. Complexity}, 13:125--179, 1997.

\bibitem{PVW}
C.~Pugh, M.~Viana, and A.~Wilkinson.
\newblock Absolute continuity of foliations.
\newblock In preparation.

\bibitem{Ro52}
V.~A. Rokhlin.
\newblock On the fundamental ideas of measure theory.
\newblock {\em A. M. S. Transl.}, 10:1--52, 1952.
\newblock Transl. from Mat. Sbornik 25 (1949), 107--150.

\bibitem{RW01}
D.~Ruelle and A.~Wilkinson.
\newblock Absolutely singular dynamical foliations.
\newblock {\em Comm. Math. Phys.}, 219:481--487, 2001.

\bibitem{Sh87}
M.~Shub.
\newblock {\em Global stability of dynamical systems}.
\newblock Springer Verlag, 1987.

\bibitem{SS85}
M.~Shub and D.~Sullivan.
\newblock Expanding endomorphisms of the circle revisited.
\newblock {\em Ergodic Theory Dynam. Systems}, 5:285--289, 1985.

\bibitem{SW00}
M.~Shub and A.~Wilkinson.
\newblock Pathological foliations and removable zero exponents.
\newblock {\em Invent. Math.}, 139:495--508, 2000.

\bibitem{Sm67}
S.~Smale.
\newblock Differentiable dynamical systems.
\newblock {\em Bull. Am. Math. Soc.}, 73:747--817, 1967.

\bibitem{Tsu05}
M.~Tsujii.
\newblock Physical measures for partially hyperbolic surface endomorphisms.
\newblock {\em Acta Math.}, 194:37--132, 2005.

\bibitem{Va07}
C.~H. V{\'a}squez.
\newblock Statistical stability for diffeomorphisms with dominated splitting.
\newblock {\em Ergodic Theory Dynam. Systems}, 27:253--283, 2007.

\bibitem{Almost}
M.~Viana.
\newblock Almost all cocycles over any hyperbolic system have nonvanishing
  {L}yapunov exponents.
\newblock {\em Ann. of Math.}, 167:643--680, 2008.

\bibitem{VY2}
M.~Viana and J.~Yang.
\newblock Contributions to the theory of maps with mostly contracting center
  direction.
\newblock In preparation.

\end{thebibliography}

\end{document}